\documentclass[12pt]{amsart}
\usepackage{amsmath}
\usepackage{amsfonts, amssymb, euscript, mathrsfs}
\usepackage{amsthm, upref}
\usepackage{graphicx}
\usepackage[usenames, dvipsnames]{color}

\usepackage{tikz}
\usepackage{enumerate}
\usepackage[tmargin=0.8in,bmargin=0.8in,lmargin=0.8in,rmargin=0.8in]{geometry}
\usepackage{aliascnt}
\usepackage{hyperref}
\usepackage{verbatim}

\allowdisplaybreaks[4]

\newcommand{\MZ}{\mathbb{Z}}

\newcommand{\BR}{\mathbb{R}}

\newcommand{\SL}{\sum\limits}

\newcommand{\al}{\alpha}
\newcommand{\be}{\beta}

\newcommand{\CF}{\mathcal F}

\newcommand{\MP}{\mathbf P}

\newcommand{\CS}{\mathcal S}

\newcommand{\CW}{\mathcal W}

\newcommand{\Oa}{\Omega}

\newcommand{\si}{\sigma}

\newcommand{\pa}{\partial}
\renewcommand{\phi}{\varphi}

\newcommand{\eps}{\varepsilon}

\newcommand{\Ra}{\Rightarrow}

\newcommand{\ol}{\overline}
\newcommand{\CM}{\mathcal M}

\newcommand{\norm}[1]{\lVert#1\rVert}
\renewcommand{\comment}[1]{}

\newcommand{\md}{\mathrm{d}}

\DeclareMathOperator{\SRBM}{SRBM}

\DeclareMathOperator{\const}{const}

\DeclareMathOperator{\SP}{SP}
\DeclareMathOperator{\CP}{CP}

\begin{document}

\theoremstyle{plain}
\newtheorem{thm}{Theorem}[section]
\newtheorem*{thmnonumber}{Theorem}
\newtheorem{lemma}[thm]{Lemma}
\newtheorem{prop}[thm]{Proposition}
\newtheorem{cor}[thm]{Corollary}
\newtheorem{open}[thm]{Open Problem}

\theoremstyle{definition}
\newtheorem{defn}{Definition}
\newtheorem{asmp}{Assumption}
\newtheorem{notn}{Notation}
\newtheorem{prb}{Problem}

\theoremstyle{remark}
\newtheorem{rmk}{Remark}
\newtheorem{exm}{Example}
\newtheorem{clm}{Claim}

\author{Andrey Sarantsev}

\title[Comparison Techniques for Competing Brownian Particles]{Comparison Techniques for Competing Brownian Particles} 

\address{Department of Statistics and Applied Probability, University of California, Santa Barbara}

\email{sarantsev@pstat.ucsb.edu}

\date{March 15, 2016. Version 27}



\begin{abstract} Consider a finite system of Brownian particles on the real line. Each particle has drift and diffusion coefficients depending on its current rank relative to other particles, as in Karatzas, Pal and Shkolnikov (2012). We prove some comparison results for these systems. As an example, we show that if we remove a few particles from the top, then the gaps between adjacent particles become stochastically larger, the local times of collision between adjacent particles become stochastically smaller, and the remaining particles shift upward, in the sense of stochastic ordering. 
\end{abstract}

\maketitle

\section{Introduction} 

\subsection{Definition of a system of competing Brownian particles} 

The topic of this paper is {\it comparison theorems} for a {\it system of $N$ competing Brownian particles with asymmetric collisions}. These results are applied in later articles \cite{MyOwn3, MyOwn5, MyOwn6, MyOwn11}. This system
$$
Y = \left(Y_1, \ldots, Y_N\right)',\ \ \ Y_k = (Y_k(t), t \ge 0),\ \ k = 1, \ldots, N,
$$
is defined as follows. Fix real numbers $g_1, \ldots, g_N$ and positive real numbers $\si_1, \ldots, \si_N$. In addition, fix real numbers $q^+_1$, $q^-_1, \ldots, q^+_N$, $q^-_N$, satisfying the following conditions:
$$
q^+_{k+1} + q^-_k = 1,\ \ k = 1, \ldots, N-1;\ \ 0 < q^{\pm}_k < 1,\ \ k = 1, \ldots, N.
$$
A one-dimensional Brownian motion with drift zero and diffusion one is called a {\it standard Brownian motion}. Let $\BR_+ := [0, \infty)$. Let $\BR_+^d$ be the $d$-dimensional positive orthant. 

\begin{defn} In the standard setting: a filtered probability space $(\Oa, \CF, (\CF_t)_{t \ge 0}, \MP)$ with the filtration satisfying the usual conditions, take i.i.d. standard $(\CF_t)_{t \ge 0}$-Brownian motions $B_1, \ldots, B_N$. Consider a continuous adapted $\BR^N$-valued process 
$$
Y = (Y(t),\ t \ge 0),\ \ \ Y(t) = (Y_1(t), \ldots, Y_N(t))',
$$
and $N-1$ continuous adapted real-valued processes
$$
L_{(k-1, k)} = (L_{(k-1, k)}(t),\ t \ge 0),\ \ k = 2, \ldots, N,
$$
with the following properties:

\medskip

(i) $Y_1(t) \le \ldots \le Y_N(t),\ \ t \ge 0$;

\medskip

(ii) if we let $L_{(0, 1)}(t) \equiv 0$ and $L_{(N, N+1)}(t) \equiv 0$ for notational convenience, then 
\begin{equation}
\label{SDEasymm}
Y_k(t) = Y_k(0) + g_kt + \si_kB_k(t) + q^+_kL_{(k-1, k)}(t) - q^-_kL_{(k, k+1)}(t),\ \ \ k = 1, \ldots, N;
\end{equation}

\medskip

(iii) for $k = 2, \ldots, N$, we have: $L_{(k-1, k)}(0) = 0$, $L_{(k-1, k)}$ is nondecreasing and can increase only when $Y_{k-1} = Y_k$.   

\medskip

Then the process $Y$ is called {\it a system of $N$ competing Brownian particles with asymmetric collisions}, with {\it drift coefficients} $g_1, \ldots, g_N$, {\it diffusion coefficients} $\si_1^2, \ldots, \si_N^2$, and {\it parameters of collision} $q^{\pm}_1,\ldots, q^{\pm}_N$. For each $k = 1, \ldots, N$, the process $Y_k = (Y_k(t), t \ge 0)$ is called the {\it $k$th ranked particle}. For $k = 2, \ldots, N$, the process $L_{(k-1, k)}$ is called the {\it local time of collision between the particles $Y_{k-1}$ and $Y_k$}. 
The {\it gap process} is defined as an $\BR^{N-1}_+$-valued process 
$$
Z = (Z(t),\ t \ge 0),\ Z(t) = \left(Z_1(t), \ldots, Z_{N-1}(t)\right)', 
$$
with 
$$
Z_k(t) = Y_{k+1}(t) - Y_k(t),\ \ k = 1, \ldots, N-1,\ \ t \ge 0.
$$
The component process $Z_k = (Z_k(t), t \ge 0)$ is called the {\it gap between the particles} $Y_k$ and $Y_{k+1}$, for $k = 1,\ldots, N-1$. 

\label{asymmdefn}
\end{defn}

This system was introduced in the paper \cite{KPS2012}; it was shown in this paper that 
it exists in the strong sense and is pathwise unique.

\subsection{An informal preview of some results of this paper} In this paper, we prove some comparison results about these (and more general) systems. 

These comparison results turn out to be important in later research on competing Brownian particles. For example, they allow us to prove a necessary and sufficient condition for a.s. avoiding triple collisions, see \cite{MyOwn3}. Also, these techniques are useful for infinite systems of competing Brownian particles, see \cite{MyOwn6, MyOwn11}. 

As a preview, let us mention a few concrete (and rather intuitive) results.

\medskip

(i) If we remove a few competing Brownian particles $Y_{M+1}, \ldots, Y_N$ from the right, the positions of the remaining particles $Y_1(t), \ldots, Y_M(t)$ at any time $t \ge 0$ shift to the right (in the sense of stochastic comparison), because they no longer feel pressure from the right, exerted by the removed particles. Moreover, the local times $L_{(k, k+1)}(t)$ stochastically decrease, and the gaps $Z_k(t)$ stochastically increase, for $k = 1, \ldots, N-1$. (Corollary~\ref{rightremovalCBP}.)

\medskip

(ii) If we shift (in the sense of stochastic comparison) initial positions $Y_k(0)$, $k = 1, \ldots, N$, of all competing Brownian particles to the right, then their positions $Y_k(t)$, at any fixed time $t \ge 0$ also shift to the right, in the sense of stochastic comparison. (Corollary~\ref{initialshiftCBP} (i).)

\medskip

(iii) If we stochastically increase the initial gaps $Z_k(0)$, $k = 1, \ldots, N-1$, between particles, then at any time $t \ge 0$ the values of the gaps $Z_k(t)$ also stochastically increase, and the local times $Y_{(k, k+1)}(t)$ stochastically decrease, for $k = 1, \ldots, N-1$.(Corollary~\ref{initialshiftCBP} (ii).) 

\medskip

(iv) If we increase the values of parameters $q^+_1, \ldots, q^+_N$, then the particles $Y_1(t), \ldots, Y_N(t)$ stochastically shift to the right. (Corollary~\ref{increaseqCBP}.) 

\medskip

We get these (and similar) results as corollaries of the two main results stated in Section 3: Theorems~\ref{main1} and~\ref{main2}. These two theorems deal with general {\it systems of competing particles}, which are generalizations of competing Brownian particles: they have arbitrary continuous driving functions $X_1(t), \ldots, X_N(t)$, in place of Brownian motions $g_1t + \si_1B_1(t), \ldots, g_Nt + \si_NB_N(t)$. 

Although these results are intuitive and natural, their proofs turn out to be very complicated and technical. Essentially, we approximate the $\BR^N$-valued function 
$$
\left(g_1t + \si_1B_1(t), \ldots, g_Nt + \si_NB_N(t)\right)'
$$
by piecewise linear functions with each piece parallel to a coordinate axis. For such piecewise linear functions, we can solve for $Y_1, \ldots, Y_N$ explicitly and compare these solutions piece by piece.

\subsection{Connection to a semimartingale reflected Brownian motion} The gaps $Z_1, \ldots, Z_{N-1}$ between particles form an $\BR^{N-1}_+$-valued process called a {\it semimartingale reflected Brownian motion} (SRBM) {\it in the orthant}. This is a process which behaves as a Brownian motion inside this orthant, and reflects on each face of the boundary $\pa\BR^{N-1}_+$ in a constant direction (not necessarily normally). We discuss this in detail in Section 2. We also provide comparison techniques for this class of processes. It was originally introduced as a {\it heavy traffic limit} for series of queues, when the intensity at each queue tends to $1$, see \cite{Har1973, Har1978, HW1987b}; see also the survey \cite{Wil1995}. For an SRBM in the orthant, as well as for related systems called {\it fluid networks}, some related results are already known, see \cite{KW1996, KR2012b, R2000, Haddad2010}. 

\subsection{Organization of the paper} The rest of the Introduction contains motivation and historical review of the concept of competing Brownian particles. In Section 2, we state all necessary definitions and provide essential notation and background. In Section 3, we formulate Theorems~\ref{main1} and~\ref{main2}; then we state and prove a few corollaries (including the ones mentioned above). Section 4 is devoted to the proofs of Theorems~\ref{main1} and~\ref{main2}. In Appendix, we state and prove some technical lemmas. 

\subsection{Motivation and historical review of competing Brownian particles}

The topic of competing Brownian particles was started in \cite{BFK2005}. Originally, the so-called {\it named particles} were considered. These are systems of $N$ particles $X_1(t), \ldots, X_N(t)$ on the real line which can swap ranks: for example, $X_1(0) < X_2(0)$ but $X_1(1) > X_2(1)$. Their dynamics can be described as follows: the $k$th leftmost particle $X_{(k)}(t)$ (which, as we say, {\it has rank $k$}), moves as a Brownian motion with drift $g_k$ and diffusion $\si_k^2$, for $k = 1, \ldots, N$. (A remark is in order: if there is a {\it tie}, that is, two or more particles occupy the same position, then it is {\it resolved in the lexicographic order}; that is, particles with smaller names are assigned smaller ranks. However, this particular way of resolving ties is not important, because the set of times when there is a tie has Lebesgue measure zero.) 

If the particles swap ranks, then their drift and diffusion coefficients are also updated: that is, a particle moves according to its {\it current} rank. More precisely, the particles are driven by the following SDE:
$$
\md X_i(t) = \SL_{k=1}^N1\left(X_i\  \mbox{has rank}\ k\  \mbox{at time}\  t\right)\left(g_k\md t + \si_k\md W_i(t)\right),
$$
where $W_1, \ldots, W_N$ are i.i.d. standard Brownian motions. The original motivation came from mathematical finance. In real world, it can be observed that stocks with smaller capitalizations have larger growth rates and larger volatilities. It is straightforward to construct a market model of $N$ stocks which is based on competing Brownian particles and which satisfies this property. Just take $g_1 > g_2 > \ldots > g_N$ and $\si_1 > \si_2 > \ldots > \si_N$, and let 
$$
e^{X_1(t)}, \ldots, e^{X_N(t)}
$$
be the capitalizations of these stocks at time $t \ge 0$. This model also allows to explain another phenomenon observed in the real world: stability of the log-log ranked market weights, see \cite{CP2010}. For other applications of this model in mathematical finance, see \cite{JR2013b, MyOwn4}. Denote by $Y_k(t)$ the particle which at time $t$ has rank $k$. The processes 
$Y_k = (Y_k(t), t \ge 0),\ k = 1, \ldots, N$, are called {\it ranked competing Brownian particles}. By definition, they satisfy
\begin{equation}
\label{comparisonofY}
Y_1(t) \le Y_2(t) \le \ldots \le Y_N(t),\ \ t \ge 0.
\end{equation}
It turns out that the ranked particles form a system from Definition~\ref{asymmdefn} with parameters of collision $q^{\pm}_k = 1/2$, $k = 1, \ldots, N$. 
For each $k = 1, \ldots, N$, the ranked particle $Y_k = (Y_k(t), t \ge 0)$ moves as a Brownian motion with drift $g_k$ and diffusion $\si_k^2$, as long as it stays away from the adjacent ranked particles $Y_{k-1}$ and $Y_{k+1}$. But when $Y_k$ and $Y_{k-1}$ collide, a {\it local time of collision} comes into play. This local time is a nondecreasing process $L_{(k-1, k)} = (L_{(k-1, k)}(t), t \ge 0)$, starting from zero ($L_{(k-1, k)}(0) = 0$), which can increase only when $Y_{k-1} = Y_k$. 

If $Y_{k-1}$ and $Y_k$ were freely moving Brownian motions, then eventually we would have $Y_{k-1}(t) > Y_k(t)$. This is not allowed by~\eqref{comparisonofY}. So the local time $L_{(k-1, k)}$ plays the role of the push which makes the ranked particle $Y_{k-1}$ stay to the left of $Y_k$. It pushes $Y_{k-1}$ to the left, and $Y_k$ to the right. 
It turns out (see \cite[Corollary 2.6]{BG2008}, \cite[Lemma 1]{Ichiba11}) that this push is distributed {\it evenly} between $Y_{k-1}$ and $Y_k$. That is, the part $\frac12 \md L_{(k-1, k)}(t)$ belongs to the particle $Y_k$ and pushes it to the right, while the part $\frac12 \md L_{(k-1, k)}(t)$ (with the minus sign) belongs to the particle $Y_{k-1}$ and pushes it to the left. Moreover, the ranked particle $Y_k$ collides not only with its left neighbor $Y_{k-1}$, but also with its right neighbor $Y_{k+1}$. So $Y_k$ experiences push from the left, when it collides with $Y_{k+1}$. This push is equal to $-\frac12 L_{(k, k+1)}(t)$, where the process $L_{(k, k+1)} = (L_{(k, k+1)}(t), t \ge 0)$ is defined similarly to $L_{(k-1, k)}$. We can summarize this in the following equation: for certain i.i.d. standard Brownian motions $B_1, \ldots, B_N$, 
\begin{equation}
\label{basicrankedSDE}
Y_k(t) = Y_k(0) + g_kt + \si_kB_k(t) + \frac12 L_{(k-1, k)}(t) - \frac12 L_{(k, k+1)}(t),\ \ k = 1, \ldots, N,\ \ t \ge 0.
\end{equation}
As before, for the sake of notational convenience, we let $L_{(0, 1)}(t) \equiv 0$ and $L_{(N, N+1)}(t) \equiv 0$. 

In the paper \cite{KPS2012}, this model was altered: the new model is, in fact, precisely the one introduced in Definition~\ref{asymmdefn}. Instead of the coefficients $1/2$ and $-1/2$ in the equation~\eqref{basicrankedSDE}, we have real-valued parameters $q^+_k$ and $-q^-_k$ for the new altered model. These numbers $q^{\pm}_k,\ k = 1, \ldots, N$, are called {\it parameters of collision}. The local time $L_{(k-1, k)}$ in the new model is split {\it unevenly} between $Y_{k-1}$ and $Y_k$: the part $q^+_k\md L_{(k-1, k)}(t)$ belongs to the ranked particle $Y_k$ and pushes it to the right, and the part $q^-_{k-1}\md L_{(k-1, k)}$ (with the minus sign) belongs to the particle $Y_{k-1}$ and pushes it to the left. As mentioned before, these parameters of collision must satisfy
\begin{equation}
\label{parametersofcollision}
q^+_k + q^-_{k-1} = 1,\ \ k = 1, \ldots, N-1;\ q^+_k > 0,\ q^-_{k} > 0,\ k = 1, \ldots, N.
\end{equation}
This new model is called {\it a system of competing Brownian particles with asymmetric collisions}. Note that for this altered model, we started from ranked particles $Y_1, \ldots, Y_N$ rather than named particles $X_1, \ldots, X_N$. If $q^+_k \ne 1/2$ or $q^-_k \ne 1/2$ for at least some $k = 1, \ldots, N$ (collisions are not symmetric), then it is not known for the general case how to define named particles $X_1, \ldots, X_N$. One can define them only up to the first moment of a {\it triple collision}: when three or more particles occupy the same position at the same time; see \cite{KPS2012, MyOwn3}.

There are other applications of systems of competing Brownian particles: they serve as scaling limits of a certain class of exclusion processes on $\MZ$, namely asymmetrically colliding random walks, see \cite{KPS2012}. They are also a discrete analogue of the \textsc{McKean-Vlasov} equation, which governs so-called {\it nonlinear diffusion processes}, where drift and diffusion coefficients at time $t$ depend not on the value of the process $X(t)$, but rather on $F_t(X(t))$, where $F_t$ is a cdf of $X(t)$. For this connection, see \cite{S2012, 4people, Jourdain2000, JR2013a}. 

\begin{rmk} We must stress that in this paper, we consider only systems of ranked competing Brownian particles $\left(Y_1, \ldots, Y_N\right)'$, rather than systems of named competing Brownian particles $\left(X_1, \ldots, X_N\right)'$. 
\end{rmk}

\begin{rmk} If two or more particles $X_1, \ldots, X_N$ collide, it is not immediately clear how to assign ranks to these particles (because they occupy the same position). Usually, we assign lower ranks to particles $X_i(t)$ with smaller indices $i$. However, this rule does not matter, because the set of times when there is a collision of at least two particles has Lebesgue measure zero a.s., see \cite[Lemma 2.3]{MyOwn6}.
\end{rmk}

Systems of competing Brownian particles were further studied in \cite{PP2008, CP2010, PS2010, IchibaThesis, Ichiba11, IK2010, IKS2013, IPS2012, JM2008, JR2013a, JR2013b, JR2014, Reygner2014, MyOwn3, MyOwn5}. There are many generalizations of this model: a {\it second-order model}, where drift and diffusion coefficients depend both on the current rank and on the name of a particle, studied in \cite{FIK2013, Ichiba11}; {\it systems of competing L\'evy particles}, when particles moves as general L\'evy processes instead of Brownian motions, see \cite{S2011}. Infinite systems of competing Brownian particles were introduced in \cite{PP2008}, and further studied in \cite{S2011, IKS2013, DemboTsai, MyOwn6}.  

In the paper \cite{KPS2012}, systems of competing Brownian particles $Y = (Y_1, \ldots, 
Y_N)'$ were used as a diffusion limit for certain interacting particle systems on $\mathbb Z$, which behave in a certain way like asymmetric simple exclusion processes. We conjecture that similar comparison results hold for these interacting particle systems.

\section{Definitions and Background}

\subsection{Notation} We denote by $I_k$ the $k\times k$-identity matrix. For a vector $x = (x_1, \ldots, x_d)' \in \BR^d$, let $\norm{x} := \left(x_1^2 + \ldots + x_d^2\right)^{1/2}$ be its Euclidean norm. For any two vectors $x, y \in \BR^d$, their dot product is denoted by $x\cdot y = x_1y_1 + \ldots + x_dy_d$. We compare vectors $x$ and $y$ componentwise: $x \le y$ if $x_i \le y_i$ for all $i = 1, \ldots, d$; $x < y$ if $x_i < y_i$ for all $i = 1, \ldots, d$; similarly for $x \ge y$ and $x > y$. We compare matrices of the same size componentwise, too. For example, we write $x \ge 0$ for $x \in \BR^d$ if $x_i \ge 0$ for $i = 1, \ldots, d$; $C = (c_{ij})_{1 \le i, j \le d} \ge 0$ if $c_{ij} \ge 0$ for all $i$, $j$. This comparison notation is also valid when $d = \infty$, that is, when we compare infinite-dimensional vectors. 

For $x \in \BR^d$ (this includes the case $d = \infty$), we let $[x, \infty) := \{y \in \BR^d\mid y \ge x\}$. We say two probability measures $\nu_1$ and $\nu_2$ on $\BR^d$ satisfy $\nu_1 \preceq \nu_2$ if for every $y \in \BR^d$ we have: $\nu_1[y, \infty) \le \nu_2[y, \infty)$. We say that $\nu_1$ {\it is stochastically dominated by} $\nu_2$, or $\nu_2$ {\it stochastically dominates} $\nu_1$, or $\nu_1$ {\it is stochastically smaller than} $\nu_2$, or $\nu_2$ {\it is stochastically larger than} $\nu_1$. The same terminology applies to $\BR^d$-valued random variables $X$, $Y$: we say that $X$ is stochastically dominated by $Y$ if the distribution of $X$ is stochastically dominated by the distribution of $Y$. 

Fix $d \ge 1$, and let $I \subseteq \{1, \ldots, d\}$ be a nonempty subset. Write its elements in the order of increase: 
$$
I = \{i_1, \ldots, i_m\},\ \ 1 \le i_1 < i_2 < \ldots < i_m \le d.
$$
For any $x \in \BR^d$, let
$[x]_I := (x_{i_1}, \ldots, x_{i_m})'$. For any $d\times d$-matrix $C = (c_{ij})_{1 \le i, j \le d}$, let 
$$
[C]_I := \left(c_{i_ki_l}\right)_{1 \le k, l \le m}.
$$
In particular, if $p = 1, \ldots, d$, then we let $[x]_p := \left(x_1, \ldots, x_p\right)'$. Let $J \subseteq \{1, \ldots, d\}$ be another nonempty subset. Write its elements in the order of increase: 
$$
J = \{j_1, \ldots, j_l\},\ \ 1 \le j_1 < j_2 < \ldots < j_l \le d.
$$
Then we denote 
$$
[C]_{IJ} := \left(c_{i_kj_s}\right)_{\substack{1 \le k \le m\\ 1 \le s \le l}}.
$$
In particular, for $I = J$ we have: $[C]_{IJ} \equiv [C]_I$. We let 
$$
\mathcal W_N := \{y = (y_1, \ldots, y_N)' \in \BR^N\mid y_1 \le \ldots \le y_N\}.
$$
The set $C([0, T], \BR^d)$ is a Banach space of continuous functions $X : [0, T] \to \BR^d$ with the norm 
\begin{equation}
\label{eq:max-norm}
\norm{X} := \max\limits_{t \in [0, T]}\norm{X(t)}.
\end{equation}

\subsection{Systems of competing particles} As mentioned in the Introduction, in this article we consider not just systems of competing Brownian particles, but more general systems, with arbitrary continuous functions instead of Brownian motions. These general systems are called {\it systems of competing particles}; they might be random or deterministic. 
Let us now define them. 

\begin{defn} Fix a continuous function $X = (X_1, \ldots, X_N)' : \BR_+ \to \BR^N$ such that $X(0) \in \mathcal W_N$. Take {\it parameters of collision}: real numbers $q^{+}_1, q^-_1, \ldots, q^{+}_N, q^-_N$ which satisfy~\eqref{parametersofcollision}. Consider a continuous function $Y = (Y_1, \ldots, Y_N)' : \BR_+ \to \mathcal W_N$, and other $N-1$ continuous functions $L_{(1, 2)}, \ldots, L_{(N-1, N)} : \BR_+ \to \BR$ such that:

\medskip

(i) $Y_k(t) = X_k(t) + q^+_kL_{(k-1, k)}(t) - q^-_kL_{(k, k+1)}(t)$ for $k = 1, \ldots, N$ and $t \ge 0$ (we let $L_{(0, 1)}(t) \equiv 0$ and $L_{(N, N+1)}(t) \equiv 0$ for notational convenience);

\medskip

(ii) $L_{(k, k+1)}(0) = 0$ for $k = 1, \ldots, N-1$;

\medskip

(iii) $L_{(k, k+1)}$ is nondecreasing for each $k = 1, \ldots, N-1$;

\medskip

(iv) $L_{(k, k+1)}$ can increase only when $Y_k(t) = Y_{k+1}(t)$; we can write this formally as the following Stieltjes integral:
$$
\int_0^{\infty}\left(Y_{k+1}(t) - Y_k(t)\right)\md L_{(k, k+1)}(t) = 0,\ \ k = 1, \ldots, N-1.
$$

\medskip

Then the function $Y$ is called a {\it system of $N$ competing particles} with {\it driving function} $X$ and {\it parameters of collisions} $(q^{\pm}_k)_{1 \le k \le N}$. The $k$th component $Y_k$ of the function $Y$ is called the {\it $k$th ranked particle}. The function $L_{(k, k+1)}$ is called the {\it collision term} between the $k$th and the $k+1$st ranked particles $Y_k$ and $Y_{k+1}$. The vector-valued function 
$L = (L_{(1, 2)}, L_{(2, 3)}, \ldots, L_{(N-1, N)})'$ is called the {\it vector of collision terms}. We say that this system {\it starts with $y$}, if $Y(0) = y$. 
\label{generalDef}
The {\it gap process} is defined as was already shown in the Introduction: this is an $\BR^{N-1}_+$-valued process 
$$
Z = (Z(t),\ t \ge 0),\ Z(t) = \left(Z_1(t), \ldots, Z_{N-1}(t)\right)', 
$$
$$
Z_k(t) = Y_{k+1}(t) - Y_k(t),\ \ k = 1, \ldots, N-1,\ \ t \ge 0.
$$
\end{defn}

Now, for the sake of completeness, we essentially rephrase Definition~\ref{asymmdefn}, tying systems of competing {\it Brownian} particles to general systems of competing particles. 

\begin{defn} Assume the standard probabilistic setting: a filtered probability space \newline $(\Oa, \CF, (\CF_t)_{t \ge 0}, \MP)$, with the filtration satisfying the usual conditions. Take i.i.d. standard $(\CF_t)_{t \ge 0}$-Brownian motions $B_1, \ldots, B_N$ and an $\CF_0$-measurable random vector $y \in \mathcal W_N$. Fix parameters of collision $q^{+}_1, q^-_1, \ldots, q^{+}_N, q^-_N$ which satisfy~\eqref{parametersofcollision}. Also, fix real numbers $g_1, \ldots, g_N$ and positive real numbers $\si_1, \ldots, \si_N$. Consider a system $Y$ of $N$ competing particles with the driving function 
$$
X = \left(X_1, \ldots, X_N\right)',\ \ X_k(t) = y_k + g_kt + \si_kB_k(t),\ \ k = 1, \ldots, N,\ \ t \ge 0,
$$
and parameters of collision $(q^{\pm}_k)_{1 \le k \le N}$. 
Then $Y$ is called a ({\it ranked}) {\it system of competing Brownian particles} with {\it drift coefficients} $g_1, \ldots, g_N$, {\it diffusion coefficients} $\si_1, \ldots, \si_N$, and {\it parameters of collision} $(q^{\pm}_k)_{1 \le k \le N}$. The standard Brownian motions $B_1, \ldots, B_N$ are called the {\it driving Brownian motions}. For each $k = 1, \ldots, N-1$, the collision term $L_{(k, k+1)}$ is called the {\it local time of collision} between $Y_k$ and $Y_{k+1}$. The vector of collision terms $L = (L_{(1, 2)}, \ldots, L_{(N-1, N)})'$ is called the {\it vector of local times}.
\label{CBPdef}
\end{defn}

Existence and uniqueness for systems of competing particles from Definition~\ref{generalDef} is proved below. (This straightforward proof is completely analogous to the proof for competing Brownian particles, which was given in \cite[subsection 2.1]{KPS2012}.) 

We can also define {\it infinite} systems of competing particles. The paper \cite{MyOwn6} deals with infinite systems of competing Brownian particles in detail, and it uses a few facts about infinite systems from this article. 

\begin{defn} Let $X_1, X_2, \ldots : \BR_+ \to \BR$ be continuous functions with $X_1(0) \le X_2(0) \le \ldots$ Take {\it parameters of collision}: real numbers $q^+_n, q^-_n,\ n = 1, 2, \ldots$, which satisfy
$$
q^+_{n+1} + q^-_n = 1,\ \ 0 < q^{\pm}_n < 1,\ \ n = 1, 2, \ldots
$$
Consider continuous functions $Y_1, Y_2, \ldots : \BR_+ \to \BR$, $L_{(1, 2)}, L_{(2, 3)}, \ldots : \BR_+ \to \BR$ such that (i), (ii), (iii) and (iv) from Definition~\ref{generalDef} are true, for $k = 1, 2, \ldots$. We let $L_{(0, 1)} \equiv 0$, as in Definition~\ref{generalDef}. Then the system $Y = (Y_1, Y_2, \ldots)$ is called an {\it infinite system of competing particles} with {\it driving function} $X = (X_1, X_2, \ldots)$ and {\it parameters of collision} $(q^{\pm}_n)_{n \ge 1}$. All other terms are defined as in Definition~\ref{generalDef}. Similarly, Definition~\ref{CBPdef} can be adapted for infinite number of Brownian particles. 
\label{InfDefCBP}
\end{defn}

Existence and uniqueness theorem is much harder to prove for infinite systems than for finite systems. Studying infinite systems of competing Brownian particles is the topic of \cite{MyOwn6}, where we prove, in particular, existence results. In this article (see Remark~\ref{rmk:inf}), we state and prove a few comparison theorems for infinite systems, {\it assuming they exist}.  

\begin{rmk} In the rest of the paper, when we use the term {\it parameters of collision}, we always assume that they satisfy condition~\eqref{parametersofcollision}. %Only in Appendix, subsection 5.2, we consider the case when some of these parameters can be equal to zero or one: the case of {\it totally asymmetric collisions}. 
\end{rmk}

\subsection{The Skorohod problem and a semimartingale reflected Brownian motion (SRBM) in the orthant}

The gap process for a system of competing Brownian particles is an $\BR^{N-1}_+$-valued process. When it is away from the boundary $\pa \BR^{N-1}_+$ (that is, when no two particles collide), the gap process behaves as an $N-1$-dimensional Brownian motion. A trickier question is what happens when this process hits the boundary (or, equivalently, when two particles collide). In fact, the gap process reflects from the boundary of the orthant $\BR^{N-1}_+$, but {\it obliquely} rather than normally. Such process is called a {\it semimartingale reflected Brownian motion} (SRBM) with {\it oblique reflection}. Let us first informally describe it, then state the formal definition. 

Fix the dimension $d \ge 1$. Take three parameters: a $d\times d$-matrix $R$ with diagonal elements equal to $1$; another $d\times d$-matrix $A$, which is symmetric and positive definite; and a vector $\mu \in \BR^d$. Denote by $S = \BR_+^d$ the positive $d$-dimensional orthant. A {\it semimartingale reflected Brownian motion} (SRBM) is an $S$-valued Markov process with the following properties:

\medskip

(i) inside the orthant $S$, it moves as a $d$-dimensional Brownian motion with drift vector $\mu$ and covariance matrix $A$;

\medskip

(ii) at each {\it face} $S_i = \{x \in S\mid x_i = 0\}$ of the boundary $\pa S$, it is reflected according to the vector $r_i$, the $i$th column of $R$;

\medskip

If $r_i = e_i$, where $e_i$ is the $i$th vector in the standard basis of $\BR^d$, then the reflection at the face $S_i$ is called {\it normal}. Otherwise, it is called {\it oblique}. 

Now, let us state a formal definition. First, we define a concept of a Skorohod problem in the orthant $\BR^d_+$. This is similar to an SRBM, but with an arbitrary continuous function in place of a Brownian motion. 

\begin{defn} Take a continuous function $X : \BR_+ \to \BR^d$ with $X(0) \in S$. A {\it solution to the Skorohod problem in the positive orthant $S$ with reflection matrix $R$ and driving function $X$} is a continuous function $Z : \BR_+ \to S$ such that there exists another continuous function $L : \BR_+ \to \BR^d$ with the following properties:

\medskip

(i) for every $t \ge 0$, we have: $Z(t) = X(t) + RL(t)$; 

\medskip

(ii) for every $i = 1, \ldots, d$, the function $L_i$ is nondecreasing, satisfies $L_i(0) = 0$ and can increase only when $Z_i(t) = 0$, that is, when $Z(t) \in S_i$. We can write the last property formally as $\int_0^{\infty}Z_i(t)\md L_i(t) = 0$.

\medskip

The function $L$ is called the {\it vector of boundary terms}, and its component $L_i$ is called the {\it boundary term}, corresponding to the face $S_i$, for $i = 1, \ldots, d$. 
\end{defn}

\begin{rmk} This definition can also be stated for a finite time horizon, that is, for functions $X, L, Z$ defined on $[0, T]$ instead of $\BR_+$. 
\end{rmk}

Now we are ready to define an SRBM in the orthant. Take the parameters $R, A, \mu$, described above. Assume the usual setting: a filtered probability space $(\Oa, \CF, (\CF_t)_{t \ge 0}, \MP)$ with the filtration satisfying the usual conditions. 

\begin{defn} Suppose $B = (B(t), t \ge 0)$ is an $(\CF_t)_{t \ge 0}$-Brownian motion in $\BR^d$ with drift vector $\mu$ and covariance matrix $A$. A solution $Z = (Z(t), t \ge 0)$ to the Skorohod problem in $S$ with reflection matrix $R$ and driving function $B$ is called a {\it semimartingale reflected Brownian motion}, or SRBM, {\it in the positive orthant} $S$  with  {\it reflection matrix} $R$, {\it drift vector} $\mu$ and {\it covariance matrix $A$}. It is denoted by $\SRBM^d(R, \mu, A)$.  The process $B$ is called the {\it driving Brownian motion}. We say that $Z$ {\it starts from} $x \in S$ if $Z(0) = x$ a.s. 
\label{SRBMdefn}
\end{defn}

The following definition describes the type of reflection matrices we are going to use in this article. An important equivalent characterization of these matrices is given in \cite[Lemma 2.1]{MyOwn3}. 

\begin{defn} A $d\times d$-matrix $R$ is called a {\it reflection nonsingular $\CM$-matrix} if it can be represented as $R = I_d - Q$, where $Q \ge 0$ has spectral radius less than $1$ and has zeros on the main diagonal. 
\end{defn}

Denote $C_S(\BR_+, \BR^d) := \{X \in C(\BR_+, \BR^d)\mid X(0) \in S\}$, and 
$C_S([0, T], \BR^d) := \{X \in C([0, T], \BR^d)\mid X(0) \in S\}$. The following existence and uniqueness result was shown in \cite{HR1981a}; see also \cite{Wil1995}. 

\begin{prop} Suppose the matrix $R$ is a reflection nonsingular $\CM$-matrix. 

\medskip

(i) For every function $X \in C_S(\BR_+, \BR^d)$, there exists a unique pair of functions $(Z, L) \in C(\BR_+, \BR^d)\times C(\BR_+, \BR^d)$ such that $Z$ is a solution to the Skorohod problem in the orthant $S$, with driving function $X$ and reflection matrix $R$, and $L$ is the corresponding vector of boundary terms. The same result is true if we change $\BR_+$ to $[0, T]$. 

\medskip

(ii) For any $T$, the following mapping is Lipschitz continuous with respect to the maximum norm in~\eqref{eq:max-norm}:
$$
\SP_T : C_S([0, T], \BR^d) \to C([0, T], \BR^d)\times C([0, T], \BR^d),\ \ \SP_T : X \mapsto (Z, L).
$$
 
\medskip

(iii) Take a drift vector $\mu \in \BR^d$, a $d\times d$ positive definite symmetric matrix $A$, and an initial condition $x \in S$. An $\SRBM^d(R, \mu, A)$, starting from $x$, exists in the strong sense and is pathwise unique. The collection of these processes for all $x \in S$ is a Feller continuous strong Markov family. 
\label{existenceSRBM} 
\end{prop}

\subsection{Connection between systems of competing particles and the Skorohod problem}

The gap process for a system of competing particles is a solution to the Skorohod problem in the orthant. In particular, the gap process for competing Brownian particles is an SRBM in the orthant.

\begin{lemma} For a system of competing particles from Definition~\ref{generalDef}, its gap process is a solution to the Skorohod problem in the orthant $\BR^{N-1}_+$ with reflection matrix
\begin{equation}
\label{R}
R = 
\begin{bmatrix}
1 & -q^-_2 & 0 & 0 & \ldots & 0 & 0\\
-q^+_2 & 1 & -q^-_3 & 0 & \ldots & 0 & 0\\
0 & -q^+_3 & 1 & -q^-_4 & \ldots & 0 & 0\\
\vdots & \vdots & \vdots & \vdots & \ddots & \vdots & \vdots\\
0 & 0 & 0 & 0 & \ldots & 1 & -q^-_{N-1}\\
0 & 0 & 0 & 0 & \ldots & -q^+_{N-1} & 1
\end{bmatrix}
\end{equation}
and driving function
\begin{equation}
\label{drivingforgaps}
\left(X_2 - X_1, X_3 - X_2, \ldots, X_N - X_{N-1}\right)'.
\end{equation}
The vector $L$ of collision terms is, in fact, the vector of boundary terms for this Skorohod problem. The matrix $R$ in~\eqref{R} is a reflection nonsingular $\CM$-matrix. 
\label{connection}
\end{lemma}

\begin{proof} Just use the property (i) from Definition~\ref{generalDef}; the gap process has the following representation:
$$
Z_k(t) = Y_{k+1}(t) - Y_k(t) = X_{k+1}(t) - X_k(t) + L_{(k, k+1)}(t) - q^+_kL_{(k-1, k)}(t) - q^-_{k+1}L_{(k+1, k+2)}(t),
$$
for $k = 1, \ldots, N-1$, $t \ge 0$. That $R$ is a reflection nonsingular $\CM$-matrix is proved in \cite{KPS2012}. 
\end{proof}

\begin{cor}
For a system of competing Brownian particles from Definition~\ref{generalDef}, its gap process is an $\SRBM^{N-1}(R, \mu, A)$, where $R$ is given by~\eqref{R}, and
\begin{equation}
\label{A}
A = 
\begin{bmatrix}
\si_1^2 + \si_2^2 & -\si_2^2 & 0 & 0 & \ldots & 0 & 0\\
-\si_2^2 & \si_2^2 + \si_3^2 & -\si_3^2 & 0 & \ldots & 0 & 0\\
0 & -\si_3^2 & \si_3^2 + \si_4^2 & -\si_4^2 & \ldots & 0 & 0\\
\vdots & \vdots & \vdots & \vdots & \ddots & \vdots & \vdots\\
0 & 0 & 0 & 0 & \ldots & \si_{N-2}^2 + \si_{N-1}^2 & -\si_{N-1}^2\\
0 & 0 & 0 & 0 & \ldots & -\si_{N-1}^2 & \si_{N-1}^2 + \si_N^2
\end{bmatrix},
\end{equation}
\begin{equation}
\label{mu}
\mu = \left(g_2 - g_1, g_3 - g_2, \ldots, g_N - g_{N-1}\right)'.
\end{equation}
The vector of local times of collision is the vector of boundary terms for this SRBM. 
\end{cor}

\begin{proof} This follows from Lemma~\ref{connection}; it was, in fact, already proved in \cite[Secton 2.1]{KPS2012}.
\end{proof}

This connection allows us to prove existence and uniqueness for systems of competing particles. Let $C_{\mathcal W}(\BR_+, \BR^N) := \{X \in C(\BR_+, \BR^N)\mid X(0) \in \mathcal W_N\}$, and $C_{\mathcal W}([0, T], \BR^N) := \{X \in C([0, T], \BR^N)\mid X(0) \in \mathcal W_N\}$. 

\begin{lemma} Fix $N \ge 2$, the number of particles. Take parameters of collision $(q^{\pm}_n)_{1 \le n \le N}$ which satisfy~\eqref{parametersofcollision}. 

\medskip

(i) For every function $X \in C_{\mathcal W}(\BR_+, \BR^N)$, there exists a unique pair of functions $(Y, L) \in C(\BR_+, \BR^N)\times C(\BR_+, \BR^N)$ such that $Y$ is a system of $N$ competing particles with driving function $X$ and parameters of collision $(q^{\pm}_n)_{1 \le n \le N}$, and $L$ is the corresponding vector of collision terms. 
The same result is true for finite time horizon $[0, T]$ in place of $\BR_+$. 

\medskip
\label{existenceCP}
(ii) For every $T > 0$, the following mapping is Lipschitz continuous with respect to the norm in~\eqref{eq:max-norm}: 
$$
\CP_T : C_{\mathcal W}([0, T], \BR^N) \to C([0, T], \BR^N)\times C([0, T], \BR^N),\ \ 
X \mapsto (Y, L).
$$

\medskip

Take drift coefficients $g_1, \ldots, g_N \in \BR$, diffusion coefficients $\si_1^2, \ldots, \si_N^2 > 0$, and a starting point $y \in \CW_N$. The ranked system of $N$ competing Brownian particles with parameters $(g_n)_{1 \le n \le N}$, $(\si_n^2)_{1 \le n \le N}$, $(q^{\pm}_n)_{1 \le n \le N}$, starting from $y$, exists in the strong sense and is pathwise unique. The collection of these processes, starting from different $y \in \mathcal W_N$, forms a Feller continuous strong Markov family. 
\end{lemma}

\begin{proof} (i) Consider the mapping $\Phi : C_{\CW}([0, T], \BR^N) \to C_S([0, T], \BR^{N-1})$, defined by
$$
\Phi : \left(X_1, \ldots, X_N\right)' \mapsto \left(X_2 - X_1, \ldots, X_N - X_{N-1}\right)'.
$$
By Lemma~\ref{connection}, the matrix $R$ from~\eqref{R} is a reflection nonsingular $\CM$-matrix. Therefore, by Proposition~\ref{existenceSRBM} (i), we can define the mapping $\SP_T$. If $Z$ is the gap process for the system $Y$, then from Lemma~\ref{connection} we get: $(Z, L) = \SP_T(\Phi(X))$. 
Recall that from the definition of the system of competing particles, we have:
$$
\begin{cases}
Y_1(t) = X_1(t) - q^-_1L_{(1, 2)}(t),\\
Y_2(t) = X_2(t) + q^+_2L_{(1, 2)}(t) - q^-_2L_{(2, 3)}(t),\\
\ldots\\
Y_N(t) = X_N(t) - q^-_NL_{(N-1, N)}(t)
\end{cases}
$$
We can find a linear combination of $Y_1, \ldots, Y_N$ which eliminates the collision terms: let
\begin{equation}
\label{alphas}
\al_1 = 1,\ \al_2 = \frac{q^-_1}{q^+_2},\ \al_3 = \frac{q^-_1q^-_2}{q^+_2q^+_3}, \ldots
\end{equation}
Then we get:
\begin{equation}
\label{eq:Z-0}
Z_0(t) \equiv \al_1X_1(t) + \ldots + \al_NX_N(t) = \al_1Y_1(t) + \ldots + \al_NY_N(t).
\end{equation}
Let $\tilde{Z} = \left(Z_0, \ldots, Z_{N-1}\right)' \in \BR^N$; then $\tilde{Z}(t) \equiv CY(t)$, where $C$ is the following $N\times N$-matrix:
$$
C = 
\begin{bmatrix}
\al_1 & \al_2 & \al_3 & \ldots & \al_N\\
-1 & 1 & 0 & \ldots & 0\\
0 & -1 & 1 & \ldots & 0\\
\vdots & \vdots & \vdots & \ddots & \vdots\\
0 & 0 & 0 & \ldots & 1
\end{bmatrix}
$$
One can show that the matrix $C$ is nonsingular. Therefore, $Y(t) = C^{-1}\tilde{Z}(t)$ for $t \ge 0$. This proves existence and uniqueness for $(Y, L)$. This proof works equally well for a finite time horizon. 

\medskip

(ii) The mapping $\CP_T$ is constructed in the proof of part (i) as a composition of the following mappings: $\SP_T$, $\Phi$, $X \mapsto Z_0$ from~\eqref{eq:Z-0}, and 
$\tilde{Z} \mapsto C^{-1}\tilde{Z} = Y$. All these mappings are Lipschitz continuous. 

\medskip

(iii) Follows from (i) and (ii). 
\end{proof}

\section{Results}

\subsection{Main results} 

Let us now state the two main results of this paper. The first result is devoted to the Skorohod problem in the orthant. It states that the solution to the Skorohod problem and the boundary terms are, in some sense, monotone with respect to the driving function and the reflection matrix. 

\begin{thm}
\label{main1}
Fix the dimension $d \ge 1$ and let $S = \BR^d_+$. 
Consider two continuous functions $X, \ol{X} : \BR_+ \to \BR^d$ such that 
$X(0),\ \ol{X}(0) \in S$, and 
\begin{equation}
\label{maincomparison}
X(0) \le \ol{X}(0),\ \ X(t) - X(s) \le \ol{X}(t) - \ol{X}(s),\ t \ge s \ge 0.
\end{equation}
Take two $d\times d$ reflection nonsingular $\CM$-matrices $R$ and $\ol{R}$ such that $R \le \ol{R}$. Let $Z$ and $\ol{Z}$ be the solutions to the Skorohod problems in the orthant $S$ with reflection matrices $R$, $\ol{R}$, and driving functions $X$, $\ol{X}$, respectively. Let $L$, $\ol{L}$ be the corresponding vectors of boundary terms. Then 
$$
Z(t) \le \ol{Z}(t),\ \ L(t) - L(s) \ge \ol{L}(t) - \ol{L}(s),\ \ t \ge s \ge 0.
$$
\end{thm}

Let us explain this informally. Suppose we make the values $X(t), t \ge 0$, and increments $X(t) - X(s),\ 0 \le s \le t$, of the driving function $X$, as well as the off-diagonal elements $r_{ij},\ i \ne j$, of the reflection matrix $R$, smaller. (The diagonal elements $r_{ii}$, $i = 1,\ldots, d$, by definition, are always equal to $1$.) Then the value $Z(t)$ to the Skorohod problem $Z$ decreases (at any fixed time $t \ge 0$), and the values of boundary terms $L_i(t)$, $i = 1, \ldots, d$, increase. 

This is what one would expect: if the driving function $X$ decreases, this will cause the ``driven function'' $Z$ also to decrease, at least until $Z$ is moving inside the orthant $S$. Indeed, $Z$ ``wants to follow'' $X$, by definition of the Skorohod problem. However, since the values $Z(t)$ of the function $Z$ become smaller at any fixed time $t \ge 0$, the process $Z$ hits the boundary more often. 

And this leads to increase in the boundary terms, which grow when $Z$ hits the boundary, and which are ``helping'' $Z$ to stay in the orthant $S$. (Recall that the driving function $X$ starts from the orthant but can leave it later.) The boundary terms $L_j(t) \ge 0$ become larger, while the off-diagonal elements $r_{ij} \le 0,\ i \ne j$, of the reflection matrix $R$ become smaller. So the terms $r_{ij}L_j(t) \le 0$ become smaller for all $i \ne j$. This observation is the core idea of the proof. 

%The term $r_{ii}L_i(t) = L_i(t)$ is the only term in decomposition
%\begin{equation}
%\label{SDE}
%Z_i = X_i + \SL_{j=1}^dr_{ij}L_j(t)
%\end{equation}
%that becomes larger, but it cannot make $Z_i$ larger than it already is, because it grows only when $Z_i = 0$, and $Z_i \ge 0$ always. 

\begin{rmk} Note that the condition that the reflection matrix $R$ has non-positive off-diagonal elements (in other words, that it is a nonsingular reflection $\CM$-matrix) is crucial. Suppose that $r_{21} > 0$. 
When $Z$ hits the face $S_1$, that is, when $Z_1(t) = 0$, the boundary term $L_1$ might increase by some increment $\md L_1(t)$. So the component $Z_2$ might get additional increase $r_{21}\md L_1(t)$. 
Consider a concrete example: two driving functions $X$ and $\ol{X}$, with 
$$
X_1(t) = -t,\ \ol{X}_1(t) = 1 - t,\ X_i(t) = \ol{X}_i(t) = 1,\ i = 2, \ldots, d.
$$
These functions satisfy the conditions of Theorem~\ref{main1}. Let $R = \ol{R}$ be a reflection nonsingular $\CM$-matrix. Let us solve the Skorohod problem in the orthant $S$ for reflection matrix $R$ and driving functions $X$ and $\ol{X}$. The function $X$ hits $S_1$ already at time $t = 0$, but $\ol{X}$ does this at time $t = 1$. So $Z_2$ gets some of this increase mentioned above before $\ol{Z}_2$ does. Actually, one can find the solutions explicitly: for $t \in [0, 1]$, 
$$
Z_2(t) = 1 + r_{21}t,\ \ol{Z}_2(t) = 1.
$$
Therefore, the statement of Theorem~\ref{main1} is not true in this case. 
\end{rmk}

%Let us describe the method of proof by \cite{Kella} and \cite{Whitt} in \cite{KW1996}: discretizing time and using that a one-step operator satisfies the fixed point theorem.

The part of Theorem~\ref{main1} concerning the functions $Z$ and $\ol{Z}$ is already known: see \cite{KR2012b, R2000, KW1996, Haddad2010}. However, our results allow to compare not just solutions to the Skorohod problem, but boundary terms as well. This comparison of boundary terms plays crucial role in some of the proofs in our subsequent paper \cite{MyOwn6}. We could not find the results about boundary terms in the existing literature; this served as a motivation for Theorem~\ref{main1}. 

The other theorem deals with systems of competing particles. Consider a system of $N$ competing particles. If we increase the values and increments of driving functions, as well as the coefficients $q^+_n,\ n = 2, \ldots, N$, then the output $Y(t)$ (positions of competing particles) will increase, too. Increasing coefficients $q^+_n,\ n = 2, \ldots, N$, has the following sense: for each $n$, at every collision between the ranked particles $Y_n$ and $Y_{n+1}$, the share of the push going to $Y_{n+1}$ (which pushes this particle to the right) increases, and the share of the push going to $Y_n$ (which pushes this particle to the left) decreases. 

\begin{thm}
\label{main2}
Fix $N \ge 2$, the number of particles. Consider two continuous functions 
$X,\ \ol{X} : \BR_+ \to \BR^N$, with $X(0), \ol{X}(0) \in \mathcal W_N$, such that
$$
X(0) \le \ol{X}(0),\ X(t) - X(s) \le \ol{X}(t) - \ol{X}(s),\ 0 \le s \le t.
$$
Fix parameters of collision $(q^{\pm}_n)_{1 \le n \le N}$ and $(\ol{q}^{\pm}_n)_{1 \le n \le N}$, such that 
$$
q^+_n \le \ol{q}^+_n,\ n = 2, \ldots, N.
$$
Consider systems $Y$ and $\ol{Y}$ of competing particles with driving functions $X$ and $\ol{X}$, and parameters of collision $(q^{\pm}_n)_{1 \le n \le N}$ and $(\ol{q}^{\pm}_n)_{1 \le n \le N}$. Then 
$$
Y(t) \le \ol{Y}(t),\ t \ge 0.
$$
\end{thm}

\subsection{Corollaries} There are many corollaries of these two main results, which are straightforward but interesting. They are used in \cite{MyOwn6}. We shall state and prove them in this subsection. 

\begin{cor} Take a $d\times d$-reflection nonsingular $\CM$-matrix $R$.
Consider two copies of an $\SRBM^d(R, \mu, A)$: $Z$ and $\ol{Z}$, starting from $Z(0)$ and $\ol{Z}(0)$ such that $Z(0) \preceq \ol{Z}(0)$. Let $L$ and $\ol{L}$ be the corresponding vectors of boundary terms. Then 
$$
Z(t) \preceq \ol{Z}(t),\ t \ge 0;
$$
$$
L(t) - L(s) \succeq \ol{L}(t) - \ol{L}(s),\ 0 \le s \le t.
$$
\end{cor}

\begin{proof} We can switch from stochastic domination $Z(0) \preceq \ol{Z}(0)$ to a.s. domination, by changing the probability space. Assume that $B = (B(t), t \ge 0)$ is a $d$-dimensional Brownian motion, starting at the origin, with drift vector $\mu$ and reflection matrix $A$. Then $Z$ and $\ol{Z}$ are solutions to the Skorohod problem in $\BR^d_+$ with driving functions $Z(0) + B(t)$, $\ol{Z}(0) + B(t)$, respectively, and reflection matrix $R$, and $L$, $\ol{L}$ are corresponding vectors of boundary terms. 
The rest follows from Theorem~\ref{main1}. 
\end{proof}

\begin{cor} Fix $N \ge 2$, the number of particles. Also, fix parameters of collision $(q^{\pm}_n)_{1 \le n\le N}$. Take two continuous functions $X, \ol{X} : \BR_+ \to \BR^N$ such that for 
\begin{equation}
\label{defofW}
W = (X_2 - X_1, \ldots, X_N - X_{N-1})',\ \ol{W} = (\ol{X}_2 - \ol{X}_1, \ldots, \ol{X}_N - \ol{X}_{N-1})',
\end{equation}
we have: 
$$
W(0) \le \ol{W}(0),\ W(t) - W(s) \le \ol{W}(t) - \ol{W}(s),\ 0 \le s \le t.
$$
Let $Y$, $\ol{Y}$ be the systems of competing particles with parameters of collision $(q^{\pm}_n)_{1 \le n \le N}$ and driving functions $X$ and $\ol{X}$, respectively. Let $Z$, $\ol{Z}$ be the corresponding gap processes, and let $L$, $\ol{L}$ be the corresponding vectors of collision terms. Then 
$$
Z(t) \le \ol{Z}(t),\ t \ge 0;\ L(t) - L(s) \ge \ol{L}(t) - \ol{L}(s),\ 0 \le s \le t.
$$
\label{cor4}
\end{cor}

\begin{proof} The functions $Z$ and $\ol{Z}$ are solutions to the Skorohod problem in the orthant $\BR^{N-1}_+$ with reflection matrix $R$ from~\eqref{R} and driving functions $W$ and $\ol{W}$, respectively. The functions $L$ and $\ol{L}$ are the corresponding vectors of boundary terms for these two Skorohod problems. Apply Theorem~\ref{main1} and finish the proof. 
\end{proof}

\begin{cor} Suppose $X : \BR_+ \to \BR^d$ is a continuous function with $X(0) \in S$. Fix a $d\times d$-reflection nonsingular $\CM$-matrix $R$. 
Take a nonempty subset $I \subseteq \{1, \ldots, d\}$ with $|I| = p$. Let $Z$ be the solution to the Skorohod problem in $S$ with reflection matrix $R$ and driving function $X$, and let $L$ be the corresponding vector of boundary terms. Also, let $\ol{Z}$ be the solution to the Skorohod problem in $\BR_+^p$ with reflection matrix $[R]_I$ and driving function $[X]_I$, and let $\ol{L}$ be the corresponding vector of boundary terms. Then 
$$
[Z(t)]_I \le \ol{Z}(t),\ t \ge 0;\ [L(t)]_I - [L(s)]_I \ge \ol{L}(t) - \ol{L}(s),\ 0 \le s \le t.
$$
\label{cor3}
\end{cor}

\begin{rmk} Corollary~\ref{cor3} has the following intuitive sense: suppose we remove a few components of the driving function. Then these (no longer existing) components do not hit zero and do not contribute  (via boundary terms) to the  decrease of the remaining components. If the component $j$ was removed but the component $i$ stayed, then in the equation for $Z_i(t)$ there is no longer the term $r_{ij}L_j(t) \le 0$. Thus, $Z_i(t)$ becomes larger. 
\end{rmk}

\begin{proof} Recall that $Z(t) \equiv X(t) + RL(t)$. For $i \in I$, $t \ge 0$,
$$
Z_i(t) = X_i(t) + \SL_{j \in I}r_{ij}L_j(t) + \SL_{j \notin I}r_{ij}L_j(t).
$$
Therefore, $[Z]_I$ is the solution of the Skorohod problem in $\BR^{p}_+$ with reflection matrix $[R]_I$ and driving function 
$$
\ol{X} = (\ol{X}_i)_{i \in I},\ \ \ \ol{X}_i(t) = X_i(t) + \SL_{j \notin I}r_{ij}L_j(t),\ \ i \in I.
$$
But $r_{ij} \le 0$ for $i \in I,\ j \in I^c$, because $R$ is a reflection nonsingular $\CM$-matrix. Moreover, each of the processes $L_j,\ j \in I^c$, is nondecreasing. Therefore, 
$$
\ol{X}_i(t) - \ol{X}_i(s) \le X_i(t) - X_i(s),\ \ 0 \le s \le t,\ \ i \in I.
$$
Apply Theorem~\ref{main1} and finish the proof. 
\end{proof}

The following corollary is a consequence (and a Brownian counterpart) of 
Corollary~\ref{cor3}. 

\begin{cor}
Take a $d\times d$ reflection nonsingular $\CM$-matrix $R$, a $d\times d$ positive definite symmetric matrix $A$, and a drift vector $\mu \in \BR^d$. Fix a nonempty subset $I \subseteq \{1, \ldots, d\}$. Let 
$$
Z = \SRBM^d(R, \mu, A),\ \ \ol{Z} = \SRBM^{|I|}([R]_I, [\mu]_I, [A]_I)
$$
such that $[Z(0)]_I$ has the same law as $\ol{Z}(0)$. Then $[Z]_I \preceq \ol{Z}$. 
\label{propcomp}
\end{cor}

\begin{cor} Let $1 < N  \le M$. Fix a continuous function $X : \BR_+ \to \BR^M$ with $X(0) \in \mathcal W_M$. Fix parameters of collision $(q^{\pm}_n)_{1 \le n \le M}$. Let $Y$ be the system of $M$ competing particles with parameters of collision $(q^{\pm}_n)_{1 \le n \le M}$ and driving function $X$. Let $\ol{Y}$ be the system of $N$ competing particles with parameters of collision $(q^{\pm}_n)_{1 \le n \le N}$ and driving function $[X]_N$. Let $Z$, $\ol{Z}$ be the corresponding gap processes, and let $L$, $\ol{L}$ be the corresponding vectors of boundary terms. Then 
\begin{equation}
\label{firstassertion}
Z_k(t) \le \ol{Z}_k(t),\ k = 1, \ldots, N-1,\ t \ge 0;
\end{equation}
\begin{equation}
\label{secondassertion}
L_k(t) - L_k(s) \ge \ol{L}_k(t) - \ol{L}_k(s),\ k = 1, \ldots, N-1,\ 0 \le s \le t;
\end{equation}
\begin{equation}
\label{thirdassertion}
Y_k(t) \le \ol{Y}_k(t),\ k = 1, \ldots, N,\ t \ge 0.
\end{equation}
\label{corremovalCP}
\end{cor}

\begin{rmk} Corollary~\ref{corremovalCP} has the following meaning: if we take a system of competing particles and remove a few particles from the right, then there is ``less pressure'' on the remaining left particles which would push them further to the left. Therefore, the gaps become wider;  there are less collisions, so the collision terms become smaller; and the remaining particles themselves shift to the right.
\end{rmk}

\begin{proof} For the system $Y$, we can write the first $N$ particles as
$$
\begin{cases}
Y_1(t) = X_1(t) - q^-_1L_{(1,2)}(t), \\
Y_2(t) = X_2(t) + q^+_2L_{(1,2)}(t) - q^-_2L_{(2,3)}(t),\\
\ldots\\
Y_N(t) = X_N(t) + q^+_NL_{(N-1, N)}(t) - q^-_NL_{(N, N+1)}(t).
\end{cases}
$$
So the vector-valued function $(Y_1, \ldots, Y_N)' = [Y]_N$ can itself be considered as a system of competing particles, with driving function
$$
\ol{X} = (X_1, X_2, \ldots, X_{N-1}, X_N - q^-_NL_{(N, N+1)}(t))'
$$
and parameters of collision $(q^{\pm}_n)_{1 \le n \le N}$. Since $L_{(N, N+1)}(0) = 0$, and $L_{(N, N+1)}$ is nondecreasing, we have:
$$
\ol{X}(0) = X(0),\ \ol{X}(t) - \ol{X}(s) \le X(t) - X(s),\ 0 \le s \le t.
$$
Therefore, by Theorem~\ref{main2}, we get: $[Y(t)]_N \le \ol{Y}(t)$, which proves~\eqref{thirdassertion}. The functions $W$ and $\ol{W}$, defined in~\eqref{defofW}, satisfy 
$$
\ol{W}(0) = W(0),\ \ol{W}(t) - \ol{W}(s) \le W(t) - W(s),\ 0 \le s \le t.
$$
Apply Corollary~\ref{cor4} to prove~\eqref{firstassertion} and~\eqref{secondassertion}. This completes the proof. 
\end{proof}

\begin{rmk}
We can also remove a few particles from the left instead of the right. We can formulate the statement analogous to Corollary~\ref{corremovalCP}. The inequalities~\eqref{firstassertion} and~\eqref{secondassertion} remain true, and the inequality~\eqref{thirdassertion} changes sign. 
\end{rmk}

If we remove particles from both the left and the right, then there are less collisions, so the remaining collision terms decrease and the remaining gaps increase. But we cannot say anything about the remaining particles themselves (whether they shift to the left or to the right). Removal of a few particles from the right eliminates some push from the right; similarly, removal of a few particles from the left eliminates some push from the left. But we cannot say which of these two effects outweighs the other one. 

\begin{cor} Fix $1 \le N_1 < N_2 \le M$. Fix a continuous function $X : \BR_+ \to \BR^M$ with $X(0) \in \mathcal W_M$. Let $Y$ be the system of $N$ competing particles with parameters of collision $(q^{\pm}_n)_{1 \le n \le M}$ and driving function $X$. Let $\ol{Y} = (\ol{Y}_{N_1}, \ldots, \ol{Y}_{N_2})'$ be the system of $N_2 - N_1 + 1$ competing particles with parameters of collision $(q^{\pm}_n)_{N_1 \le n \le N_2}$ and driving function $(X_{N_1}, \ldots, X_{N_2})'$. Let $Z = (Z_1, \ldots, Z_{M-1})'$ and $\ol{Z} = (Z_{N_1}, \ldots, Z_{N_2-1})'$ be the corresponding gap processes, and let \label{corremovalCP2}
$$
L = (L_{(1, 2)}, \ldots, L_{(M-1, M)})',\ \ol{L} = (\ol{L}_{(N_1, N_1+1)}, \ldots, \ol{L}_{(N_2-1, N_2)})',
$$
be the vectors of collision terms. Then 
$$
Z_k(t) \le \ol{Z}_k(t),\ k = N_1, \ldots, N_2-1,\ t \ge 0;
$$
$$
L_{(k, k+1)}(t) - L_{(k, k+1)}(s) \ge \ol{L}_{(k, k+1)}(t) - \ol{L}_{(k, k+1)}(s),\ k = N_1, \ldots, N_2-1,\ 0 \le s \le t.
$$
\end{cor}

The rest of the corollaries deal with competing Brownian particles. 
The first of these corollaries is a Brownian counterpart of Corollary~\ref{corremovalCP}. It says that if you remove a few competing Brownian particles from the right, then the remaining particles shift to the right, the local times of collisions decrease, and the gaps increase. This corollary was mentioned in the Introduction, subsection 1.2. 

\begin{cor} Fix $1 < N \le M$. Take a system $Y$ of $M$ competing Brownian particles with parameters $(g_k)_{1 \le k \le M}$, $(\si_k^2)_{1 \le k \le M}$, $(q^{\pm}_k)_{1 \le k \le M}$, starting from $y \in \mathcal W_M$. Let $B_1, \ldots, B_M$ be the corresponding driving Brownian motions. Take another system $\ol{Y}$ of $N$ competing Brownian particles with parameters $(g_k)_{1 \le k \le N}$, $(\si_k^2)_{1 \le k \le N}$, $(q^{\pm}_k)_{1 \le k \le N}$, starting from $[y]_N$, with driving Brownian motions $B_1, \ldots, B_N$. Let $Z$, $\ol{Z}$ be the corresponding gap processes, and let $L$, $\ol{L}$ be the corresponding vectors of collision local times. Then 
$$
Y_k(t) \le \ol{Y}_k(t),\ k = 1, \ldots, N,\ t \ge 0;
$$
$$
Z_k(t) \le \ol{Z}_k(t),\ k = 1, \ldots, N-1,\ t \ge 0;
$$
$$
L_{(k, k+1)}(t) - L_{(k, k+1)}(s) \ge \ol{L}_{(k, k+1)}(t) - \ol{L}_{(k, k+1)}(s),\ k = 1, \ldots, N-1,\ 0 \le s \le t.
$$
\label{rightremovalCBP}
\end{cor}

The next corollary is a Brownian counterpart of Corollary~\ref{corremovalCP2}. It says that if you remove a few competing Brownian particles from the right and from the left simultaneously, then the local times of collisions decrease, and the gaps increase. 

\begin{cor} Fix $1 \le N_1 < N_2 \le M$. Take a system $Y$ of $M$ competing Brownian particles with parameters $(g_k)_{1 \le k \le M}$, $(\si_k^2)_{1 \le k \le M}$, $(q^{\pm}_k)_{1 \le k \le M}$, starting from $y \in \mathcal W_M$. Let $B_1, \ldots, B_M$ be the corresponding driving Brownian motions. Take another system $\ol{Y} = (\ol{Y}_{N_1}, \ldots, \ol{Y}_{N_2})'$ of $N_2 - N_1 + 1$ competing Brownian particles with parameters $(g_k)_{N_1 \le k \le N_2}$, $(\si_k^2)_{N_1 \le k \le N_2}$, $(q^{\pm}_k)_{N_1 \le k \le N_2}$, starting from $(y_{N_1}, \ldots, y_{N_2})'$, with driving Brownian motions $B_{N_1}, \ldots, B_{N_2}$. Let $Z = (Z_1, \ldots, Z_{M-1})'$, $\ol{Z} = (\ol{Z}_{N_1}, \ldots, \ol{Z}_{N_2})'$ be the corresponding gap processes, and let $L = (L_{(1, 2)}, \ldots, L_{(M-1, M)})'$, $\ol{L} = (L_{(N_1, N_1+1)}, \ldots, L_{(N_2-1, N_2)})'$ be the corresponding vectors of collision terms. Then 
$$
Z_k(t) \le \ol{Z}_k(t),\ k = N_1, \ldots, N_2-1,\ t \ge 0;
$$
$$
L_{(k, k+1)}(t) - L_{(k, k+1)}(s) \ge \ol{L}_{(k, k+1)}(t) - \ol{L}_{(k, k+1)}(s),\ k = N_1, \ldots, N_2-1,\ 0 \le s \le t.
$$
\label{removalCBP2}
\end{cor}

\begin{rmk} Corollaries~\ref{corremovalCP}, ~\ref{corremovalCP2}, ~\ref{rightremovalCBP} and~\ref{removalCBP2} can be generalized for the case of infinite particle systems, when $M = \infty$. Recall that we introduced infinite systems of competing particles (and including competing Brownian particles) in Definition~\ref{InfDefCBP}. Again, here we do not prove existence of these infinite systems; we state these corollaries, assuming these systems exist. The proofs are the same as for finite $M$, with only trivial adjustments. 
\label{rmk:inf}
\end{rmk}

The following corollary was also mentioned in the Introduction, subsection 1.2. 

\begin{cor} Take two systems, $Y$ and $\ol{Y}$, of $N$ competing Brownian particles with parameters $(g_k)_{1 \le k \le N}$, $(\si_k^2)_{1 \le k \le N}$, $(q^{\pm}_k)_{1 \le k \le N}$. Suppose these two systems have the same driving Brownian motions. Let $Z$, $\ol{Z}$ be the corresponding gap processes, and let $L$, $\ol{L}$ be the corresponding vectors of collision terms. 

\medskip

(i) If $Y(0) \le \ol{Y}(0)$, then $Y(t) \le \ol{Y}(t),\ t \ge 0$. 

\medskip

(ii) If $Z(0) \le \ol{Z}(0)$, then $Z(t) \le \ol{Z}(t),\ t \ge 0$, and $L(t) - L(s) \ge \ol{L}(t) - \ol{L}(s),\ 0 \le s \le t$. 
\label{initialshiftCBP}
\end{cor}

The last two corollaries show how to compare systems of competing Brownian particles in case of the change in drift coefficients or parameters of collision. The first of these corollaries tells that if you increase $q^+_1, \ldots, q^+_N$, the whole system will shift to the right. 

\begin{cor} Consider two systems $Y$ and $\ol{Y}$ of $N$ competing Brownian particles with common drift and diffusion coefficients $(g_k)_{1 \le k \le N}$, $(\si_k^2)_{1 \le k \le N}$, but different parameters of collision $(q^{\pm}_k)_{1 \le k \le N}$, $(\ol{q}^{\pm}_k)_{1 \le k \le N}$,
such that $\ol{q}_n^+ \ge q^+_n,\ n = 1, \ldots, N$. Suppose $Y(0) = \ol{Y}(0)$ and the driving Brownian motions are the same for these two systems. Then
$$
Y(t) \le \ol{Y}(t),\ t \ge 0.
$$
\label{increaseqCBP}
\end{cor}

\begin{proof} Let $B_1, \ldots, B_N$ be the driving Brownian motions for these systems. Then $Y$ and $\ol{Y}$ are systems of competing particles with parameters of collision $(q^{\pm}_n)_{1 \le n \le N}$, $(\ol{q}^{\pm}_n)_{1 \le n \le N}$, and the same driving function
$$
X(t) = \left(Y_1(0) + g_1t + \si_1 B_1(t), \ldots, Y_N(0) + g_Nt + \si_NB_N(t)\right)'.
$$
Apply Theorem~\ref{main2} and finish the proof.
\end{proof}

The following corollary shows how to use the drift coefficients for comparison. 

\begin{cor} Consider two systems $Y$ and $\ol{Y}$ of $N$ competing Brownian particles with common diffusion coefficients $(\si_k^2)_{1 \le k \le N}$ and parameters of collision $(q^{\pm}_n)_{1 \le n \le N}$, but with different drift coefficients $(g_n)_{1 \le n \le N}$, $(\ol{g}_n)_{1 \le n \le N}$. Suppose $Y(0) = \ol{Y}(0)$ and the driving Brownian motions are the same for these two systems. Let $Z$ and $\ol{Z}$ be the corresponding gap processes, and let $L$ and $\ol{L}$ be the corresponding vectors of collision terms. 

\medskip

(i) If $g_k \le \ol{g}_k,\ k = 1,\ldots, N$, then $Y(t) \le \ol{Y}(t),\ t \ge 0$. 

\medskip

(ii) If $g_{k+1} - g_k \le \ol{g}_{k+1} - \ol{g}_k,\ k = 1, \ldots, N-1$, then 
$$
Z(t) \le \ol{Z}(t),\ t \ge 0;\ \ L(t) - L(s) \ge \ol{L}(t) - \ol{L}(s),\ 0 \le s \le t.
$$
\end{cor}

\begin{proof} Let $B_1, \ldots, B_N$ be the driving Brownian motions for these systems. Then $Y$ and $\ol{Y}$ are systems of competing particles with parameters of collision $(q^{\pm}_n)_{1 \le n \le N}$ and driving functions 
$$
X(t) = \left(Y_1(0) + g_1t + \si_1 B_1(t), \ldots, Y_N(0) + g_Nt + \si_NB_N(t)\right)',
$$
$$
\ol{X}(t) = \left(Y_1(0) + \ol{g}_1t + \si_1 B_1(t), \ldots, Y_N(0) + \ol{g}_Nt + \si_NB_N(t)\right)'.
$$

\medskip

(i) We have: $X(t) - X(s) \le \ol{X}(t) - \ol{X}(s),\ 0 \le s \le t$, and $X(0) = \ol{X}(0)$. Apply Theorem~\ref{main2} and finish the proof.

\medskip

(ii) This statement immediately follows from Corollary~\ref{cor4}. 
\end{proof}

In each of the last five corollaries, if we remove the requirement that the driving Brownian motions must be the same, then we get stochastic  comparison instead of pathwise comparison. 

We can also compare stationary distributions for gap processes of systems of competing Brownian particles. From previous research, \cite{BFK2005, Ichiba11, KPS2012}, \cite[Section 2]{MyOwn6}, we know that if a stationary distribution exists, then it is unique. Let us give one result, which is used in \cite{MyOwn6}. 

\begin{cor}
\label{cor:comparison-of-SD}
Fix $1 \le N_1 < N_2 \le M$. Take a system $Y$ of $M$ competing Brownian particles with parameters $(g_k)_{1 \le k \le M}$, $(\si_k^2)_{1 \le k \le M}$, $(q^{\pm}_k)_{1 \le k \le M}$. Take another system $\ol{Y} = (\ol{Y}_{N_1}, \ldots, \ol{Y}_{N_2})'$ of $N_2 - N_1 + 1$ competing Brownian particles with parameters $(g_k)_{N_1 \le k \le N_2}$, $(\si_k^2)_{N_1 \le k \le N_2}$, $(q^{\pm}_k)_{N_1 \le k \le N_2}$. If both systems have stationary distributions $\pi$ and $\ol{\pi}$ for their gap processes: $Z$ and $\ol{Z}$, and if 
$$
(z_1, \ldots, z_M)' \sim \pi,\ \ (\ol{z}_{N_1}, \ldots, \ol{z}_{N_2})' \sim \ol{\pi},
$$
then
$$
(z_{N_1}, \ldots, z_{N_2})' \preceq (\ol{z}_{N_1}, \ldots, \ol{z}_{N_2})'.
$$
\end{cor}

\begin{proof} Consider versions of the systems $Y$ and $\ol{Y}$ when all particles start from zero, and assume they have the same driving Brownian motions, as in Corollary~\ref{removalCBP2}. From this Corollary~\ref{removalCBP2}, we have: 
$$
\left(Z_{N_1}(t), \ldots, Z_{N_2}(t)\right)' \le \left(\ol{Z}_{N_1}(t), \ldots, \ol{Z}_{N_2}(t)\right)',\ \ t \ge 0.
$$
From \cite[Section 2]{MyOwn6}, we have: 
$$
Z(t) \Ra \pi\ \ \mbox{and}\ \ \ol{Z}(t) \Ra \ol{\pi}\ \ \mbox{as}\ \ t \to \infty.
$$
The rest is trivial.
\end{proof}

\section{Proofs of Theorems~\ref{main1} and~\ref{main2}}

\subsection{Outline of the proofs} We prove Theorems~\ref{main1} and~\ref{main2} by approximating the general continuous driving functions by ``simple'' functions, which are defined as follows. 

\begin{defn} A continuous function $f : [0, T] \to \BR^d$ is called {\it regular} if it is piecewise linear with each piece parallel to one of the coordinate axes; that is, if there exist a partition 
$0 = t_0 < t_1 < \ldots < t_N = T$ and numbers $\al_1, \ldots, \al_N \in \BR$, $j_1, \ldots, j_N \in \{1, \ldots, d\}$ such that for $k = 1, \ldots, N$, we have:
$$
f(t) = f\left(t_{k-1}\right) + \al_ke_{j_k}(t - t_{k-1}),\ t_{k-1} \le t \le t_k.
$$
Two regular functions $f$ and $\ol{f}$ are called {\it coupled} if the partition $t_0, \ldots, t_N$ and the indices $j_1,\ldots, j_N$ are the same for them. 
\label{defRRegular}
\end{defn} 

We make three observations:

\medskip

(i) Any continuous function $X : [0, T] \to \BR^d$ can be uniformly approximated by regular functions. This is proved in Lemma~\ref{regular}. Moreover, we show that a pair of continuous functions $X$ and $\ol{X}$ which satisfy~\eqref{maincomparison} can be uniformly on $[0, T]$ approximated by a pair of coupled regular functions so that within each pair two regular functions also satisfy~\eqref{maincomparison}.  This is proved in Lemma~\ref{regular1}. 

\medskip

(ii) All the objects we are considering in this article (the solution to the Skorohod problem in the orthant, boundary terms in the Skorohod problem, the system of competing particles, the gap process, the vector of collision terms) continuously depend on the corresponding driving functions; see Proposition~\ref{existenceSRBM}(ii) and Lemma~\ref{existenceCP}(ii). Thus, we can prove Theorems~\ref{main1} and~\ref{main2} just for regular driving functions. 

\medskip

(iii) In Lemmas~\ref{memorylessSP} and~\ref{memorylessCP}, we show that solutions to the Skorohod problem and systems of competing particles are ``memoryless'': if you take a moment $t > 0$, then their behavior after this moment depends only on their current position and future dynamics of the driving function. This is very similar to Markov property (although the concepts of the Skorohod problem and competing particles are deterministic, not random).  This allows us to consider driving regular functions (and the solutions) piece by piece. 

\medskip

The goal of these three observations is Lemma~\ref{spanning}. It shows that Theorems~\ref{main1} and~\ref{main2} can be reduced to the case when the driving functions are not just piecewise linear, but exactly linear, with the directional vector parallel to one of the axes. And since they are coupled, this axis is the same for both functions. That is, we can consider 
\begin{equation}
\label{reglin}
X(t) = x + \al e_it,\ \ol{X}(t) = \ol{x} + \ol{\al}e_it,
\end{equation}
where $\al, \ol{\al} \in \BR$, $i = 1,\ldots, d$. The condition~\eqref{maincomparison} for these functions is equivalent to
\begin{equation}
\label{compreglin}
x \le \ol{x},\ \ \al \le \ol{\al}.
\end{equation}

But for regular linear driving functions as in~\eqref{reglin}, we can actually solve the Skorohod problem explicitly, and find the solution $Z$ and and the vector of boundary terms $L$ in exact form. This is done in Lemma~\ref{explicitSkorohod}. We can do the same for the system $Y$ of competing particles: Lemma~\ref{explicitCP}. Then we can manually compare the solutions $Z$ and $\ol{Z}$ of the Skorohod problem, and the vectors $L$ and $\ol{L}$ of boundary terms, or (if we are considering systems of competing particles) $Y$ and $\ol{Y}$. This completes the proof of Theorems~\ref{main1} and~\ref{main2}. 

The rest of this section is organized as follows. 

In subsection 4.2, we state and prove the technical results mentioned above: (i) approximation of continuous driving functions by regular functions; (ii) continuous dependence on driving functions (actually, these are already stated above as Proposition~\ref{existenceSRBM}(ii) and Lemma~\ref{existenceCP}(ii)); (iii) the memoryless property. In subsection 4.3, we explicitly solve the Skorohod problem for regular driving functions in Lemma~\ref{explicitSkorohod} and find the solution together with the boundary terms. In subsection 4.4, we do the same for a system of competing particles in Lemma~\ref{explicitCP}. In subsections 4.5 and 4.6, we prove Theorems~\ref{main1} and~\ref{main2} for regular linear driving functions. This completes the proof. 

\subsection{Auxillary results} {\it Observation (i): approximation by regular driving functions.}  

\begin{lemma} Fix $T \ge 0$ and take a continuous function $X : [0, T] \to \BR^d$. Then there exists a sequence $(X^{(n)})_{n \ge 1}$ of regular functions $[0, T] \to \BR^d$ which uniformly converges to $X$ on $[0, T]$. 
\label{regular}
\end{lemma}

\begin{proof} First, let us show how to define $X^{(1)}$, the first function from this sequence. Let 
$$
t_i := \frac{Ti}d,\ i = 0, \ldots, d. 
$$
Split the interval $[0, T]$ into $d$ subintervals with equal length:
$I_i := \left[t_{i-1}, t_i\right]$, $i = 1,\ldots, d$. 
On the $i$th subinterval $I_i$, define the function $X^{(1)}$ as follows: 
$$
X^{(1)}(t) = X^{(1)}\left(t_{i-1}\right) + \left(X_i\left(T\right) - X_i\left(0\right)\right)\frac{t - t_{i-1}}{t_i - t_{i-1}}e_i,\ i = 1, \ldots, d,\ \ t_{i-1} \le t \le t_i.
$$
Then $X^{(1)}(0) = X(0)$ and $X^{(1)}(T) = X(T)$. During the time interval $I_i$, only the $i$th component of the function $X^{(1)}$ is changing; other components stay constant. The $i$th component $X^{(1)}_i$ is moving between $X_i(0)$ and $X_i(T)$. So
$$
\left|X_i^{(1)}(t) - X_i(0)\right| \le \left|X_i(T) - X_i(0)\right|,\ \ t \in [0, T].
$$
Therefore, 
$$
\norm{X^{(1)}(t) - X(0)} \le \norm{X(T) - X(0)},\ \ t \in [0, T],
$$
and
$$
\norm{X^{(1)}(t) - X(t)} \le \norm{X(T) - X(0)} + \max\limits_{0 \le t \le T}
\norm{X(t) - X(0)} \le 2\max\limits_{0 \le t \le T}
\norm{X(t) - X(0)}.
$$
Now, let us show how to define $X^{(n)}$ for each $n = 1, 2, \ldots$ 
Let $s_k := kT/n,\ k = 0, \ldots, n$. Split $[0, T]$ into $n$ equal subintervals $J_k = [s_{k-1}, s_k],\ k = 1, \ldots, n$. Perform the same construction of $X^{(1)}$ for each of these small subintervals in place of $[0, T]$. Then we get a continuous function $X^{(n)}$ on $[0, T]$ such that
$$
X^{(n)}(s_k) = X(s_k),\ \ k = 0, \ldots, n.
$$
For $t \in J_k$, we have:
$$
\norm{X^{(n)}(t) - X(t)} \le 2\!\!\max\limits_{s_{k-1} \le t \le s_k}\norm{X(t) - X\left(s_{k-1}\right)}. 
$$
Therefore, 
\begin{equation}
\label{5467}
\max\limits_{0 \le t \le T}\norm{X^{(n)}(t) - X(t)} \le 2\max\limits_{k=1, \ldots, n}\max\limits_{s_{k-1} \le t \le s_k}\norm{X\bigl(t\bigr) - X\bigl((k-1)T/n\bigr)}.
\end{equation}
But the function $X$ is uniformly continuous on $[0, T]$. Therefore, the right-hand side of~\eqref{5467} tends to zero as $n \to \infty$. Thus, the sequence of regular functions $(X^{(n)})_{n \ge 1}$ uniformly converges to $X$. 
\end{proof}

We will call the sequence constructed in Lemma~\ref{regular} {\it the standard approximating sequence}. 

\begin{lemma} Fix $T \ge 0$ and take two continuous functions $X, \ol{X} : [0, T] \to \BR^d$ such that 
$$
X(0) \le \ol{X}(0);\ \ X(t) - X(s) \le \ol{X}(t) - \ol{X}(s),\ \ 0 \le s \le t \le T.
$$
Then there exist two sequences $(X^{(n)})_{n \ge 1},\ (\ol{X}^{(n)})_{n \ge 1}$ of regular functions $[0, T] \to \BR^d$ such that:

\medskip

(i) $X^{(n)} \to X,\ \ol{X}^{(n)} \to \ol{X}$ uniformly on $[0, T]$ as $n \to \infty$;

\medskip

(ii) for every $n \ge 1$, the functions $X^{(n)}$ and $\ol{X}^{(n)}$ are coupled; 

\medskip

(iii) $X^{(n)}(0) \le \ol{X}^{(n)}(0)$ and $X^{(n)}(t) - X^{(n)}(s) \le \ol{X}^{(n)}(t) - \ol{X}^{(n)}(s)$ for all $0 \le s \le t \le T$. 
\label{regular1}
\end{lemma}

\begin{proof} Construct two standard approximating sequences as in the proof of Lemma~\ref{regular}. Let us show that 
$$
X^{(1)}(t) - X^{(1)}(s) \le \ol{X}^{(1)}(t) - \ol{X}^{(1)}(s),\ 0 \le s \le t.
$$
Indeed, $X^{(1)}$ and $\ol{X}^{(1)}$ are linear on each $[(k-1)T/d, kT/d]$, and 
$$
X^{(1)}\left(\frac{kT}d\right) - X^{(1)}\left(\frac{(k-1)T}d\right) \le \ol{X}^{(1)}\left(\frac{kT}d\right) - \ol{X}^{(1)}\left(\frac{(k-1)T}d\right).
$$
The proof is similar for $X^{(n)}$ and $\ol{X}^{(n)}$ instead of $X^{(1)}$ and $\ol{X}^{(1)}$. 
\end{proof}

\medskip

{\it Observation (ii): continuous dependence.} We have the following continuity results: Proposition~\ref{existenceSRBM} and Lemma~\ref{existenceCP}. In addition, we have approximation results: Lemmata~\ref{regular} and~\ref{regular1}. These allow us to substantially narrow the class of driving functions. Let us state this as a separate lemma.

\begin{lemma} If Theorems~\ref{main1} and~\ref{main2} are true for coupled regular driving functions, then they are true in the general case.
\label{intermediate}
\end{lemma}

\medskip

{\it Observation (iii): memoryless property.} This allows us to further narrow the scope of driving functions: to take coupled regular {\it linear} driving functions. 

\begin{lemma} Fix $d  \ge 1$. Take a continuous function $X : \BR_+ \to \BR^d$ with $X(0) \in S = \BR^d_+$ and a $d\times d$-reflection nonsingular matrix $R$. Let $Z$ be the solution of the Skorohod problem in $S$ with reflection matrix $R$ and driving function $X$. Let $L$ be the vector of boundary terms. Fix $T \ge 0$. For $t \ge 0$, let 
$$
X_T(t) = X(T + t) - X(T) + Z(T),
$$
$$
L_T(t) = L(T + t) - L(T),\ Z_T(t) = Z(T + t).
$$
Then $Z_T$ is the solution of the Skorohod problem with reflection matrix $R$ and driving function $X_T$, and $L_T$ is the corresponding vector of boundary terms. 
\label{memorylessSP}
\end{lemma}

\begin{proof} It suffices to check the definition: first, we need to prove that
$$
Z_T(t) = X_T(t) + RL_T(t),\ t \ge 0.
$$
This follows from 
$$
Z(t + T) = X(t + T) + RL(t + T)\ \ \mbox{and}\ \ Z(T) = X(T) + RL(T).
$$
Also, we need to mention that $L_T = ((L_T)_1, \ldots, (L_T)_d)'$ has each component $(L_T)_i$, $i = 1,\ldots, d$, nondecreasing, $(L_T)_i(0) = 0$, and 
$$
\int_0^{\infty}(Z_T)_i(t)\md(L_T)_i(t) = \int_0^{\infty}Z_i(T+t)\md L_i(T+t) = 
\int_T^{\infty}Z_i(s)\md L_i(s) = 0.
$$
\end{proof}

Let us state a similar property for systems of competing particles. The proof is similar to the previous one and is omitted. 

\begin{lemma} Fix $N \ge 2$. Assume $Y$ is a system of $N$ competing particles with driving function $X : \BR_+ \to \BR^N$ and parameters of collision $(q^{\pm}_n)_{1 \le n \le N}$. Let $L$ be the corresponding vector of collision terms. Fix $T \ge 0$. For $t \ge 0$, let 
$$
X_T(t) = X(T + t) - X(T) + Y(T),
$$
$$
L_T(t) = L(T + t) - L(T),\ Y_T(t) = Y(T + t).
$$
Then the function $Y_T$ is a system of $N$ competing particles with driving function $X_T$ and the same parameters of collision, and $L_T$ is the corresponding vector of collision terms. 
\label{memorylessCP}
\end{lemma}

\begin{rmk}
\label{infmemoryless}
The memoryless property also holds true for infinite systems of competing particles from Definition~\ref{InfDefCBP}. The proof is the same, with obvious adjustments. 
\end{rmk} 

The memoryless property allows us to narrow the class of driving functions to {\it regular linear} functions, that is, functions of the type~\eqref{reglin}. 

\begin{lemma} If Theorems~\ref{main1} and~\ref{main2} are true for coupled regular linear driving functions as in~\eqref{reglin}, satisfying~\eqref{compreglin}, they are true in the general case. 
\label{spanning}
\end{lemma}

\begin{proof} By Lemma~\ref{intermediate}, it suffices to show these theorems for coupled regular driving  functions. For example, let us prove Theorem~\ref{main1} for coupled regular driving functions $X$ and $\ol{X}$; Theorem~\ref{main2} is proved similarly. Let $0 = t_0 < t_1 <  \ldots < t_N = T$ and $j_1, \ldots, j_N$ be the common parameters for these functions, as in Definition~\ref{defRRegular}. The restrictions 
$$
\left.X\right|_{[t_0, t_1]},\ \left.\ol{X}\right|_{[t_0, t_1]}
$$
are coupled regular linear functions. Assuming we proved Theorem~\ref{main1} for such driving functions, we have: 
$$
Z(t) \le \ol{Z}(t),\ t \ge 0;\ L(t) - L(s) \ge \ol{L}(t) - \ol{L}(s),\ 0 \le s \le t \le t_1.
$$
In particular, we have: $Z(t_1) \le \ol{Z}(t_1)$. But 
$t \mapsto Z(t + t_1)$ is the solution of the Skorohod problem with reflection matrix $R$ and driving function $t \mapsto X(t + t_1) - X(t_1) + Z(t_1)$;  a similar statement is true for $t \mapsto \ol{Z}(t + t_1)$. And 
$$
L(t + t_1) - L(t_1),\ \ol{L}(t + t_1) - \ol{L}(t_1),\ 0 \le t \le t_2 - t_1.
$$
are the corresponding vectors boundary terms for these Skorohod problems. The functions
\begin{equation}
\label{newF}
X(t + t_1) - X(t_1) + Z(t_1)\ \ \mbox{and}\ \ \ol{X}(t + t_1) - \ol{X}(t_1) + \ol{Z}(t_1)
\end{equation}
are coupled regular linear driving functions on $[0, t_2 - t_1]$.  They also satisfy conditions of Theorem~\ref{main1}. Indeed, 
$$
 \left.X(t + t_1) - X(t_1) + Z(t_1)\right|_{t=0} = Z(t_1) \in S,\ 
\left.\ol{X}(t + t_1) - \ol{X}(t_1) + \ol{Z}(t_1)\right|_{t=0} = \ol{Z}(t_1) \in S,
$$
and for $0 \le s \le t \le t_2 - t_1$ we have:
\begin{align*}
(X(t + t_1) &- X(t_1) + Z(t_1)) - (X(s + t_1) - X(t_1) + Z(t_1)) = X(t + t_1) - X(s + t_1) \\ & \le \ol{X}(t + t_1) - \ol{X}(s + t_1) = (\ol{X}(t + t_1) - \ol{X}(t_1) + \ol{Z}(t_1)) - (\ol{X}(s + t_1) - \ol{X}(t_1) + \ol{Z}(t_1)).
\end{align*}
Therefore, applying Theorem~\ref{main1} for coupled regular linear driving functions~\eqref{newF}, we get: 
$$
Z(t + t_1) \le \ol{Z}(t + t_1),\ 0 \le t \le t_2 - t_1,
$$
$$
L(t + t_1) - L(s + t_1) \ge \ol{L}(t + t_1) - \ol{L}(s + t_1),\ 0 \le s \le t \le t_2 - t_1.
$$
Similarly, moving to the next interval $[t_2, t_3]$, etc., we can show that for every $k = 1, \ldots, N$, 
\begin{equation}
\label{9067}
Z(t) \le \ol{Z}(t),\ t \in [t_{k-1}, t_k],
\end{equation}
\begin{equation}
\label{345}
L(t + t_{k-1}) - L(s + t_{k-1}) \ge \ol{L}(t + t_{k-1}) - \ol{L}(s + t_{k-1}),\ 0 \le s \le t \le t_k - t_{k-1}.
\end{equation}
We can equivalently write~\eqref{9067} as 
$$
Z(t) \le \ol{Z}(t),\ t \in [0, T],
$$
and~\eqref{345} as
\begin{equation}
\label{346}
L(t) - L(s) \ge \ol{L}(t) - \ol{L}(s),\ t_{k-1} \le s \le t \le t_k,\ k = 1, \ldots, N.
\end{equation}
Now, let us show that 
$$
L(t) - L(s) \ge \ol{L}(t) - \ol{L}(s),\ 0 \le s \le t \le T.
$$
This is done just by summing the inequalities~\eqref{346}: find $k, l = 1, \ldots, N$ such that 
$$
t_{k-1} \le s \le t_k \le \ldots \le t_l \le t \le t_{l+1}.
$$
Then we have:
$$
\begin{cases}
L(t) - L(t_l) \ge \ol{L}(t) - \ol{L}(t_l)\\
L(t_l) - L(t_{l-1}) \ge \ol{L}(t_l) - \ol{L}(t_{l-1})\\
\ldots\\
L(s) - L(t_{k-1}) \ge \ol{L}(s) - \ol{L}(t_{k-1})
\end{cases}
$$
Sum these inequalities and finish the proof. 
\end{proof}

\subsection{Exact solutions of the Skorohod problem for regular linear driving functions}

Fix the dimension $d \ge 1$, and recall that $S = \BR^d_+$ is the positive $d$-dimensional orthant. Let 
\begin{equation}
\label{Ilinear}
X(t) = x + \al e_it,\ 0 \le t \le T,
\end{equation}
be a regular linear driving function. Here, $x \in S$, $\al \in \BR$ and $i = 1,\ldots, d$. Take a reflection nonsingular $\CM$-matrix $R$. In this subsection, we find the explicit solution $Z$ (and the vector $L$ of boundary terms) for the Skorohod problem in the orthant $S$ with  reflection matrix $R$ and driving function $X$. 

Let us first describe the behavior of this solution informally. The solution $Z$ ``wants'' to move along with the driving function $X$. However, if $X$ gets out of the orthant $S$, then $Z$ ``is not allowed'' out of the orthant; the boundary terms push it back to $S$. 

{\it Case 1.} $\al \ge 0$. Then $X$ does not get out of $S$. This is a trivial case: the boundary terms $L_i$ stay zero:  $L(t) \equiv 0$, and the solution $Z$ exactly follows the driving function $X$: $Z(t) \equiv X(t)$. 

\medskip

{\it Case 2.} $\al < 0$ and $x_i = 0$. Then the driving function $X$ is moving along the axis $x_i$ in the negative direction, starting from the face $S_i$ of the boundary $\pa S$. The solution $Z$ of the Skorohod problem ``wants'' to move in tandem with $X$, which means that it ``wants'' to cross this face $S_i$. However, it cannot do this, since it must be in the orthant. Therefore, it stays at this face. The boundary term $L_i$ increases: this term ``counters the influence'' of the driving function $X$, which ``wants'' to take $Z$ out of the orthant. This increase in $L_i$ also influences other components $Z_j,\ j \ne i$, of $Z_i$, through reflection matrix $R$ (or, more precisely, through the elements $r_{ij} \le 0,\ j \ne i$). Therefore, if $Z$ moves on the face $S_i$, this contributes to {\it decrease} of other components $Z_j,\ j \ne i$. Let 
\begin{equation}
\label{defofset}
I(t) = \{j = 1, \ldots, d\mid Z_j(t) = 0\}.
\end{equation}
Suppose $j \in I(0)$. Then $Z_j$ was originally zero, and it ``wants'' to decrease because of the increase in $L_i$. But $Z_j$ cannot decrease further, because $Z(t)$ must stay in the orthant. Therefore, the boundary term $L_j$ starts to increase, to ``counter'' the influence of $L_i$. This can be said of all $j \in I(0)$. If, however, $j \notin I(0)$, then $Z_j(0) > 0$, and so $Z_j$ ``is allowed'' to decrease, so the boundary term $L_j$ stays zero. 

Let us summarize this: for $j \in I(t)$, the boundary term $L_j$ increases, and $Z_j(t) = 0$; for $j \notin I(t)$, the boundary term $L_j(t) = 0$, and $Z_j$ decreases. This description is accurate until some new component $Z_j$ hits zero; another way to say this is when the set-valued function $I$ jumps upward.  Denote this moment by $\tau_1$. Then, using the memoryless property from Lemma~\ref{memorylessSP}, we repeat the same, starting from $\tau_1$. Let $\tau_2$ be the next jump moment of the function $I$, etc. Between any of these two moments, the function $I$ is constant. There will be no more than $d$ of these moments, because the function $I$ increases at every jump, and $i \in I(0)$, but 
$I(t) \subseteq \{1, \ldots, d\}$. 

\medskip

{\it Case 3.} $x_i > 0$ and $\al < 0$. Then $X$ moves to the boundary and hits it at some moment $\tau_1 = x_i/|\al|$. The solution $Z$ ``wants'' to move in tandem with $Z$. Until $\tau_1$, however, the solution $Z$ does not need to be pushed inside the orthant $S$ by boundary terms, so this is also a trivial case: $L(t) \equiv 0$, $Z(t) \equiv X(t)$. If $\tau_1 \ge T$, then the time-horizon is earlier than hitting moment of the boundary, and this completes the description of $Z$ and $L$. If $\tau_1 < T$, then we use the memoryless property and start from $\tau_1$; we are back in Case 2. 

\medskip

Now, let us formulate the result rigorously. 

\begin{lemma} Let $R$ be a $d\times d$ reflection nonsingular $\CM$-matrix. Let $X$ be given by~\eqref{Ilinear}. Let $Z$ be the solution to the Skorohod problem in the orthant $S$ with reflection matrix $R$ and driving function $X$. Let $L$ be the corresponding vector of boundary terms. Then $Z$ and $L$ are given by the following formulas. 

\medskip

(i) If $\al \ge 0$, then $Z(t) \equiv X(t)$ and $L(t) \equiv 0$.

\medskip

(ii) If $\al < 0$, and $x_i = 0$, then:

\smallskip

(a) $Z$ is nondecreasing, $L$ is nondecreasing, the set-valued function $I$ defined in~\eqref{defofset} is nondecreasing.

\smallskip

(b) There exists a sequence $0 = \tau_0 < \tau_1 < \ldots < \tau_K = T$ of moments such that on each $[\tau_{l-1}, \tau_l)$, $I(t)$ is constant, and 
\begin{equation}
\label{recurrenceforI}
\tau_l := \inf\{t > \tau_{l-1}\mid I(t) \ne I\left(\tau_{l-1}\right)\}\wedge T.
\end{equation}
We use the convention $\inf\varnothing = +\infty$. At each moment $\tau_l,\ l = 1,\ldots, K-1$, the function $I$ jumps and increases.  

\smallskip

(c) For $t \in [\tau_{l-1}, \tau_l]$, letting $J := I(\tau_{l-1})$, we have:
\begin{equation}
\label{formulaforZ}
[Z(t)]_J = 0;\ \ [Z(t)]_{J^c} = [Z(\tau_{l-1})]_{J^c} + |\al|[R]_{J^cJ}[R]_{J}^{-1}[e_i]_{J}(t - \tau_{l-1}),
\end{equation}
\begin{equation}
\label{formulaforL}
[L(t)]_J = [L(\tau_{l-1})]_J + |\al|[R]_{J}^{-1}(t - \tau_{l-1});\ \ [L(t)]_{J^c} = [L(\tau_{l-1})]_{J^c}.
\end{equation}

\medskip

(iii) If $\al < 0$, and $x_i > 0$, then $Z$ is nondecreasing, $L$ is nondecreasing, the set-valued function $I$ from~\eqref{defofset} is nondecreasing, and there exists a sequence $0 = \tau_0 < \tau_1 < \ldots < \tau_K = T$ of moments such that on each $[\tau_{l-1}, \tau_l)$, $I(t) \equiv I(\tau_{l-1}) =: J$ is constant, on $[0, \tau_1]$ we have:
$$
Z(t) \equiv X(t),\ \ L(t) \equiv 0,
$$
and on $[\tau_{l-1}, \tau_l]$, $l = 2, \ldots, k$, the functions $Z$ and $L$ are given by the formulas~\eqref{formulaforZ} and~\eqref{formulaforL}. As in case (ii), at each moment $\tau_l,\ l = 1,\ldots, K-1$, the function $I$ jumps and increases. 
\label{explicitSkorohod}
\end{lemma}

\begin{proof} The case (i) is straightforward. Let us show (ii). Using the memoryless property and induction by $l$, we can assume w.l.o.g. that $\tau_l = 0$, that is, $l = 0$: it suffices to consider only the first interval $[0, \tau_1]$ of linearity. We have: $x_i = 0$, that is, $i \in I(0) = J$. We can write the main equation governing $Z$ and $L$, 
$$
Z(t) = X(t) + RL(t),
$$
in the block form:
\begin{equation}
\label{systemsystem}
\begin{cases}
[Z(t)]_J = [X(t)]_J + [R]_J[L(t)]_J + [R]_{JJ^c}[L(t)]_{J^c}\\
[Z(t)]_{J^c} = [X(t)]_{J^c} + [R]_{J^cJ}[L(t)]_J + [R]_{J^c}[L(t)]_{J^c}
\end{cases}
\end{equation}
But $[X(t)]_{J^c} = [x + \al e_it]_{J^c} = [x]_{J^c}$, because $i \in J$. Also, $[X(t)]_J = \al[e_i]_Jt$, because $[x]_J = 0$. Now it is straightforward to check that the functions $Z(t)$ and $L(t)$ given by
$$
[Z(t)]_J = 0,\ \ [Z(t)]_{J^c} = |\al|[R]_{J^cJ}[R]_J^{-1}t + [x]_{J^c},
$$
$$
[L(t)]_{J^c} = 0,\ \ [L(t)]_J = |\al|[R]_J^{-1}[e_i]_Jt,
$$
satisfy the system~\eqref{systemsystem}. Let us now verify that for $j = 1, \ldots, d$, the boundary term $L_j$ can grow only when $Z_j = 0$. This follows from the fact that 
$$
Z_j(t) \equiv 0,\ j \in J;\ L_j(t) \equiv 0,\ j \in J^c.
$$
The next step is to check that $L$ is nondecreasing and $Z$ is nonincreasing on $[0, \tau_{1}]$. Indeed, by Lemma~\ref{P1} $[R]_J^{-1} \ge 0$, and $[e_i]_J \ge 0$, so 
\begin{equation}
\label{compar}
|\al|[R]_{J}^{-1}[e_i]_J \ge 0.
\end{equation}
Therefore, $L$ is nondecreasing on $[0, \tau_1]$. Next, $R$ is a reflection nonsingular $\CM$-matrix, so off-diagonal elements of $R$ (in particular, all elements of $[R]_{J^cJ}$) are nonpositive. From this and~\eqref{compar} it follows that 
$$
|\al|[R]_{J^cJ}[R]_J^{-1}[e_i]_J \le 0.
$$
So $Z$ is nonincreasing on $[0, \tau_1]$. We have the formula
$$
\tau_{1} := \inf\{t \ge 0\mid I(t) \ne I(0)\}\wedge T,
$$
so $\tau_1$ is the first moment (no later than the time horizon $T$) when $Z$ hits ``new'' parts of the boundary, and the function $I$ increases. 
If this moment comes later than $T$, then we let $\tau_1 = T$. 
By definition of $\tau_1$, we have: $I(0) \subsetneq I(\tau_1)$. So the set-valued function $I$ is constant on $[0, \tau_1)$, but increases by a jump at $\tau_1$. Part (iii) follows from (ii) and the memoryless property.
\end{proof}

\subsection{Exact formulas for a system of competing particles with a regular linear driving function}

Let us now do a similar calculation as in the previous subsection, but for a system of competing particles instead of a Skorohod problem. First, let us informally describe the dynamics of these particles. Recall that the driving function is given by~\eqref{Ilinear}. 

Without loss of generality, assume $\al > 0$. The case $\al = 0$ is trivial ($Y(t) \equiv X(t) \equiv x$ and $L(t) \equiv 0$), and the case $\al < 0$ can be reduced to $\al > 0$ by the following lemma. (The proof is trivial and is omitted.)

\begin{lemma} Suppose $Y = (Y(t), t \ge 0)$ is a system of $N$ competing particles with parameters of collision $(q^{\pm}_k)_{1 \le k \le N}$ and driving function $X$. Then the  following $\BR^N$-valued process
$$
\tilde{Y} = \left(\tilde{Y}_1, \ldots, \tilde{Y}_N\right)',\ \ \tilde{Y}_n(t) := -Y_{N-n+1}(t),\ \ n = 1, \ldots, N,
$$
is also a system of $N$ competing particles, with parameters of collision $(\tilde{q}^{\pm}_n)_{1 \le n \le N}$ and driving function $(-X_N, \ldots, -X_1)$, where
$$
\tilde{q}^+_n = q^-_{N - n + 1},\ \tilde{q}^-_n = q^+_{N - n + 1},\ \ n = 1, \ldots, N.
$$
\label{inversionCP}
\end{lemma}

A system of competing particles involves  colliding particles, and ``asymmetric collisions'' means that they ``have different mass''. We can rewrite the expression 
$$
X(t) = x + \al e_it
$$
in the coordinate form:
$$
X_i(t) = x_i + \al t,\ X_j(t) = x_j,\ j \ne i.
$$
This means that the $i$th ranked particle ``wants'' to move to the right with speed $\al$, and all other particles ``want'' to stay motionless. But when the particles, say with ranks $i$ and $i+1$, collide, they move together to the right with a new speed (smaller than $\al$). The collision term for particles $Y_i$ and $Y_{i+1}$ starts to increase linearly from zero when they first collide. All other particles stay motionless. When these two particles hit, say, the $i+2$nd particle $Y_{i+2}$, then these three particles stick together and move to the right. The collision terms $L_{(j, j+1)}$ for all other pairs of adjacent particles $Y_j, Y_{j+1}$ stay zero. Indeed, even if $Y_j(t) = Y_{j+1}(t)$, but $Y_j$ and $Y_{j+1}$ are not moving, then no collision term is required to keep them in order: $Y_j(t) \le Y_{j+1}(t)$. But the collision term $L_{(i+1, i+2)}$ starts to increase, and the collision term $L_{(i, i+1)}$ continues to increase.

In other words, at any time $t$ there is a set 
\begin{equation}
\label{7801}
I(t) = \{j = i, \ldots, N \mid Y_j(t) = Y_i(t)\}
\end{equation}
of particles which are moving together with $Y_i$ to the right at this moment $t$. Since these particles satisfy
$$
Y_1(t) \le Y_2(t) \le \ldots \le Y_N(t),
$$
the set $I(t)$ has the form 
$$
I(t) = \{i, i+1, \ldots, k(t)\}
$$
for some $k(t) = i, \ldots, N$. The speed of this movement depends on $k(t)$. When these moving particles hit a new particle $Y_{k(t)+1}$, then the set $I$ increases by a jump. So we have a sequence of moments of hits:
$$
0 = \tau_0 < \tau_1 < \ldots < \tau_K = T.
$$
At any interval between these moments, $I(t)$ is constant, the particles $Y_j,\ j \in I(t)$ move to the right, and all other particles do not move. 

Now, let us do the precise calculation. 

\begin{lemma} There exists a sequence of moments
$$
0 = \tau_0 < \tau_1 < \ldots < \tau_K := T
$$
such that on each $[\tau_{l-1}, \tau_l)$, the set-valued function $I$ defined in~\eqref{7801} is constant, but it jumps and increases at each $\tau_l$ (except maybe $\tau_K = T$). On each $[\tau_{l-1}, \tau_l)$, define
$$
\be_l = \al\left[1 + \frac{q^-_i}{q^+_{i+1}} + \frac{q^-_iq^-_{i+1}}{q^+_{i+1}q^+_{i+2}} + \ldots + \frac{q^-_iq^-_{i+1}\ldots q^-_{k_l - 1}}{q^+_{i+1}q^+_{i+2}\ldots q^+_{k_l}}\right]^{-1}, 
$$
$k_l \equiv k(t)$ for $t \in [\tau_{l-1}, \tau_l)$. Then we have:
\begin{equation}
Y_j(t) = \const,\ j \in I^c(t);\ \ \mbox{and}\ \ 
\label{543}
Y_j(t) \equiv Y_i(t) = Y_i(\tau_{l-1}) + \be_l(t - \tau_{l-1}),\ j \in I(t).
\end{equation}
The moment $\tau_l$ is defined as 
$$
\tau_l = \inf\{t \ge \tau_{l-1}\mid I(t) \ne I(\tau_{l-1})\}\wedge T.
$$
\label{explicitCP}
As before, we use the convention $\inf\varnothing = +\infty$. 
\end{lemma}

\begin{proof} Similarly to the previous subsection, we can use the memoryless property and induction  by $l$ to assume that $l = 0$. Assume $I(0) = \{i, \ldots, k_0\}$, so initially the ``leading'' particle $i$ was at the same position as particles with ranks $i+1, \ldots, k_0$. Note that we care only about particles with ranks larger than $i$, because the particle with rank $i$ is moving to the right. Even if, say, the particle with rank $i-1$ occupied the same position initially as the particle with rank $i$, they will not interact: the particle $Y_i$, together with $Y_{i+1}, \ldots, Y_{k_0}$, will move rightward and``leave'' the idle particle $i-1$ at its place. So we have: on $[0, \tau_1]$, 
$$
L_{(1, 2)}(t) = \ldots = L_{(i-1, i)}(t) = L_{(k_0, k_0+1)}(t) = \ldots = L_{(N-1, N)}(t) = 0,
$$
and $Y_1, \ldots, Y_{i-1}, Y_{k_0+1},\ldots, Y_N$ are constant on this time interval. The dynamics of the particles $Y_i, \ldots, Y_{k_0}$ on $[0, \tau_1]$ is described as follows: 
$$
\begin{cases}
Y_i(t) = Y_{i+1}(t) = \ldots = Y_{k_0}(t), \\
Y_i(t) = x_i + \al t - q^-_iL_{(i, i+1)}(t),\\
Y_{i+1}(t) = x_{i+1} + q^+_{i+1}L_{(i, i+1)}(t) - q^-_{i+1}L_{(i+1, i+2)}(t), \\
\ldots \\
Y_{k_0}(t) = x_{k_0} + q^+_{k_0}L_{(k_0-1, k_0)}(t).
\end{cases}
$$
But $x_i = x_{i+1} = \ldots = x_{k_0}$, because $Y_i(0) = Y_{i+1}(0) = \ldots = Y_{k_0}(0)$ (the initial positions of particles with ranks $i, i+1, \ldots, k_0$ are the same). We can solve this system: multiplying the third equation in the system above by $q^-_i/q^+_{i+1}$, the fourth by $q^-_iq^-_{i+1}/q^+_{i+1}q^+_{i+2}$, etc. and add these equations. We get 
the equation~\eqref{543}. Since $Y_i(t)$ is an increasing function, it might hit $Y_{k_0+1}(0) = x_{k_0+1}$ before the time horizon $T$. (If it does not, there is nothing else to prove.) Then $\tau_1$ is this hitting moment. The set-valued function $I$ is constant on $[0, \tau_1)$ but jumps at $\tau_1$. Using the memoryless property and induction, we repeat this proof starting from $\tau_1$ time instead of $0$. Since the function $I$ increases at every $\tau_l$, and it can take set values which contain $\{i\}$ and are contained in $\{i, \ldots, N\}$, there will be at most $N+1$ induction steps. 
\end{proof}

\subsection{Proof of Theorem~\ref{main1}} Take driving functions as in~\eqref{reglin} satisfying~\eqref{compreglin}. Let 
$$
\tau_0 := 0, \tau_1, \ldots, \tau_K := T
$$
be the sequence of moments described in Lemma~\ref{explicitSkorohod}, and let $\ol{\tau}_0 := 0, \ol{\tau}_1, \ldots, \ol{\tau}_{\ol{K}} := T$ be the corresponding sequence of moments for the driving function $\ol{X}$ instead of $X$. Arrange all these moments in the increasing order: 
$$
\rho_0 := 0 < \rho_1 := \tau_1\wedge\ol{\tau}_1 < \rho_2 < \ldots < \rho_M := T.
$$
Then it suffcies to show the theorem for $t \le \rho_1$. Indeed, suppose that we prove this, then we can use the memoryless property for Skorohod problems and prove this for $\rho_1 \le t \le \rho_2$, then for $\rho_2 \le t \le \rho_3$, etc. Extending the result from $[0, \rho_1]$ to $[0, T]$ requires reasoning analogous to the argument in proof of Lemma~\ref{spanning}. On $[0, \rho_1]$, we know explicit expressions for $Z$, $\ol{Z}$, $L$ and $\ol{L}$ from Lemma~\ref{explicitSkorohod}. Let $I(t)$ be the set-valued function defined in Lemma~\ref{explicitSkorohod}, and $\ol{I}(t)$ be the same function for $\ol{X}$ instead of $X$. Consider the following cases.

\medskip

{\it Case 1.} $0 \le \al \le \ol{\al}$. Then $Z(t) \equiv X(t)$, $\ol{Z}(t) \equiv \ol{X}(t)$, $L(t) \equiv \ol{L}(t) \equiv 0$, and the statement is obvious. 

\medskip

{\it Case 2.} $\al \le 0 \le \ol{\al}$. Then $Z$ is nonincreasing, $\ol{Z} = \ol{X}$ is nondecreasing, $\ol{L}(t) \equiv 0$, and $L$ is nondecreasing. Therefore, $\ol{Z} - Z$ is nondecreasing, and $L - \ol{L}$ is nonincreasing; the rest is trivial. 

\medskip

{\it Case 3.} $\al \le \ol{\al} \le 0$, and $x_i > 0$. Since $x \le \ol{x}$, we have: $\ol{x}_i > 0$; the rest is similar to Case 1. 

\medskip

{\it Case 4.} $\al \le \ol{\al} \le 0$, and $x_i = 0$, $\ol{x}_i > 0$. Then $I(0) \supsetneq \ol{I}(0)$, and on $[0, \rho_1)$ we have: $\ol{L}(t) \equiv 0$, $L$ is nondecreasing, so 
$$
L(t) - L(s) \ge \ol{L}(t) - \ol{L}(s),\ 0 \le s \le t \le \rho_1.
$$
Furthermore, $\ol{Z}(t) \equiv \ol{X}(t) = \ol{x} + \ol{\al}e_it$. And $Z(t)$ is given by: 
$$
Z_j(t) = 0 \le \ol{Z}_j(t),\ j \in I(0); 
$$
$Z(t)$ is nonincreasing, so for $j \notin I(0)$ we have: $\ol{Z}_j(t) = \ol{Z}_j(0) = \const$. Thus, 
$$
Z_j(t) \le Z_j(0) \le \ol{Z}_j(0) = \ol{Z}_j(t).
$$

\medskip

{\it Case 5.} $\al \le \ol{\al} \le 0$, and $x_i = \ol{x}_i = 0$. This is the most difficult case. We again have: $\ol{I}(0) \subseteq I(0)$, and on $[0, \rho_1]$ we get: 

\medskip

{\it Case 5.1.} $j \in \ol{I}(0)$. Then $j \in I(0)$, so $Z_j(t) \equiv \ol{Z}_j(t) \equiv 0$; therefore, $Z_j(t) - Z_j(s) \le \ol{Z}_j(t) - \ol{Z}_j(s),\ \ 0 \le s \le t$. Furthermore, 
$$
[\ol{L}(t)]_{\ol{I}(0)} = |\ol{\al}|[\ol{R}]_{\ol{I}(0)}^{-1}[e_i]_{\ol{I}(0)}t, \ \ \ [L(t)]_{I(0)} = |\al|[R]_{I(0)}^{-1}[e_i]_{I(0)}t.
$$
Applying Lemma~\ref{P5} to $J = I(0)$, we get that $[R]_{I(0)}$ is a reflection nonsingular $\CM$-matrix. Applying Lemma~\ref{P1} to $[R]_{I(0)}$ instead of $R$ and $J = \ol{I}(0)$, we get: 
\begin{equation}
\label{777077}
[R]_{\ol{I}(0)}^{-1} \le \left[[R]^{-1}_{I(0)}\right]_{\ol{I}(0)}.
\end{equation}
Also, $[e_i]_{\ol{I}(0)} = [[e_i]_{I(0)}]_{\ol{I}(0)} \ge 0$, and $|\ol{\al}| \le |\al|$. Since $R \le \ol{R}$, we have: $[R]_{\ol{I}(0)} \le [\ol{R}]_{\ol{I}(0)}$. Both $[R]_{\ol{I}(0)}$ and $[\ol{R}]_{\ol{I}(0)}$ are reflection nonsingular $\CM$-matrices of the same size, so by Lemma~\ref{P3} we have: 
\begin{equation}
\label{2310}
[R]^{-1}_{\ol{I}(0)} \ge [\ol{R}]^{-1}_{\ol{I}(0)} \ge 0.
\end{equation}
In addition, by Lemma~\ref{P7} we have:
\begin{equation}
\label{9926}
\left[[R]^{-1}_{I(0)}[e_i]_{I(0)}\right]_{\ol{I}(0)} \ge \left[[R]_{I(0)}^{-1}\right]_{\ol{I}(0)}\left[[e_i]_{I(0)}\right]_{\ol{I}(0)}
\end{equation}
Combining~\eqref{777077}, ~\eqref{2310}, ~\eqref{9926} and the fact that $|\ol{\al}| \le |\al|$, we get: for $0 \le s \le t \le \rho_1$, 
\begin{align*}
[L(t)]_{\ol{I}(0)} - [L(s)]_{\ol{I}(0)} &= |\al|\left[[R]^{-1}_{I(0)}[e_i]_{I(0)}\right]_{\ol{I}(0)}(t-s) \ge |\al|\left[[R]_{I(0)}^{-1}\right]_{\ol{I}(0)}\left[[e_i]_{I(0)}\right]_{\ol{I}(0)}(t-s) \\ & \ge |\ol{\al}|[R]^{-1}_{\ol{I}(0)}[e_i]_{\ol{I}(0)}(t-s) \ge |\ol{\al}|[\ol{R}]^{-1}_{\ol{I}(0)}[e_i]_{\ol{I}(0)}(t-s) = [\ol{L}(t)]_{\ol{I}(0)} - [\ol{L}(s)]_{\ol{I}(0)}.
\end{align*}
In other words, for $j \in \ol{I}(0)$, 
$$
L_j(t) - L_j(s) \ge \ol{L}_j(t) - \ol{L}_j(s),\ \ 0 \le s \le t \le \rho_1.
$$

\medskip

{\it Case 5.2.} $j \in I(0)\setminus \ol{I}(0)$. Then $Z_j(t) = 0 \le \ol{Z}_j(t)$. Now, $L_j$ is always nondecreasing, and $\ol{Z}_j > 0$, so $\ol{L}_j \equiv 0$. Thus, 
$$
L_j(t) - L_j(s) \ge 0 = \ol{L}_j(t) - \ol{L}_j(s),\ 0 \le s \le t \le \rho_1.
$$

\medskip

{\it Case 5.3.} $j \notin \ol{I}(0)$. Then $j \notin I(0)$. Let 
$$
I^c(0) := \{1, \ldots, d\}\setminus I(0),\ \ \ol{I}^c(0) := \{1, \ldots, d\}\setminus \ol{I}(0).
$$
The components of $Z$ and $\ol{Z}$ corresponding to the sets $I^c(0), \ol{I}^c(0)$ of indices, respectively, have the following dynamics:
$$
\begin{cases}
[Z(t)]_{I^c(0)} = [Z(0)]_{I^c(0)} + |\al|[R]_{I^c(0)I(0)}[R]^{-1}_{I(0)}[e_i]_{I(0)}t,\\
[\ol{Z}(t)]_{\ol{I}^c(0)} = [\ol{Z}(0)]_{\ol{I}^c(0)} + |\ol{\al}|[\ol{R}]_{\ol{I}^c(0)\ol{I}(0)}[\ol{R}]^{-1}_{\ol{I}(0)}[e_i]_{\ol{I}(0)}t.
\end{cases}
$$
Since $R$ and $\ol{R}$ are reflection nonsingular $\CM$-matrices, and $R \le \ol{R}$, we have: 
\begin{equation}
\label{712}
r_{ij} \le \ol{r}_{ij} \le 0,\ i \ne j.
\end{equation}
In particular, this is true for $i \in I^c(0),\ j \in I(0)$, as well as for $i \in \ol{I}^c(0),\ j \in \ol{I}(0)$. But $I(0) \supseteq \ol{I}(0)$, and so $I^c(0) \subseteq \ol{I}^c(0)$. Therefore, 
\begin{align*}
[\ol{Z}(t)]_{I^c(0)} &= [\ol{Z}(0)]_{I^c(0)} + |\ol{\al}|t\left[[\ol{R}]_{\ol{I}^c(0)\ol{I}(0)}[\ol{R}]^{-1}_{\ol{I}(0)}[e_i]_{\ol{I}(0)}\right]_{I^c(0)} \\ & = [\ol{Z}(0)]_{I^c(0)} - |\ol{\al}|t\left[[-\ol{R}]_{\ol{I}^c(0)\ol{I}(0)}[\ol{R}]^{-1}_{\ol{I}(0)}[e_i]_{\ol{I}(0)}\right]_{I^c(0)}.
\end{align*}
It follows from~\eqref{712} that 
\begin{equation}
\label{45121}
0 \le [-\ol{R}]_{\ol{I}^c(0)\ol{I}(0)} \le [-R]_{\ol{I}^c(0)\ol{I}(0)}.
\end{equation}
Also, $[\ol{R}]^{-1}_{\ol{I}(0)} \ge 0$, $[e_i]_{\ol{I}(0)} \ge 0$. By Lemma~\ref{P4}, 
\begin{equation}
\label{45122}
\left[[-\ol{R}]_{\ol{I}^c(0)\ol{I}(0)}[\ol{R}]^{-1}_{\ol{I}(0)}[e_i]_{\ol{I}(0)}\right]_{I^c(0)} =
[-\ol{R}]_{I^c(0)\ol{I}(0)}[\ol{R}]^{-1}_{\ol{I}(0)}[e_i]_{\ol{I}(0)}.
\end{equation}
By Lemma~\ref{P6} and inequalities~\eqref{45121} and~\eqref{45122},
\begin{equation}
\label{45123}
[-\ol{R}]_{I^c(0)\ol{I}(0)}[\ol{R}]_{\ol{I}(0)}^{-1}[e_i]_{\ol{I}(0)} \le 
[-R]_{I^c(0)\ol{I}(0)}[R]^{-1}_{\ol{I}(0)}[e_i]_{\ol{I}(0)}.
\end{equation}
Since $\ol{I}(0) \subseteq I(0)$, applying Lemma~\ref{P1}, we get:
\begin{equation}
\label{45124}
0 \le [R]^{-1}_{\ol{I}(0)} \le \left[[R]^{-1}_{I(0)}\right]_{\ol{I}(0)},
\end{equation}
Therefore, applying Lemma~\ref{P6} again, and using that $[e_i]_{\ol{I}(0)} = \left[[e_i]_{I(0)}\right]_{\ol{I}(0)}$, we have:
\begin{equation}
\label{45125}
[-R]_{I^c(0)\ol{I}(0)}[R]^{-1}_{\ol{I}(0)}[e_i]_{\ol{I}(0)} 
\le [-R]_{I^c(0)\ol{I}(0)}\left[[R]^{-1}_{I(0)}\right]_{\ol{I}(0)}\left[[e_i]_{I(0)}\right]_{\ol{I}(0)}.
\end{equation}
By Lemma~\ref{P2} (applied twice)
\begin{equation}
\label{45126}
[-R]_{I^c(0)\ol{I}(0)}\left[[R]^{-1}_{I(0)}\right]_{\ol{I}(0)}\left[[e_i]_{I(0)}\right]_{\ol{I}(0)} \le 
[-R]_{I^c(0)I(0)}[R]_{I(0)}^{-1}[e_i]_{I(0)}.
\end{equation}
Combining~\eqref{45123}, ~\eqref{45124}, ~\eqref{45125} and~\eqref{45126}, we get:
$$
0 \le \left[[-\ol{R}]_{\ol{I}^c(0)\ol{I}(0)}[\ol{R}]^{-1}_{\ol{I}(0)}[e_i]_{\ol{I}(0)}\right]_{I^c(0)} \le [-R]_{I^c(0)I(0)}[R]^{-1}_{I(0)}[e_i]_{I(0)}.
$$
But we also have: $|\al| \ge |\ol{\al}| \ge 0$. So 
$$
0 \le \left[[-\ol{R}]_{\ol{I}^c(0)\ol{I}(0)}[\ol{R}]^{-1}_{\ol{I}(0)}[e_i]_{\ol{I}(0)}\right]_{I^c(0)}|\ol{\al}|t \le [-R]_{I^c(0)I(0)}[R]^{-1}_{I(0)}[e_i]_{I(0)}|\al|t.
$$
Finally, we get: 
\begin{align*}
[\ol{Z}(t)]_{I^c(0)} &\ge [\ol{Z}(0)]_{I^c(0)} - |\al|t[-R]_{I^c(0)I(0)}[R]^{-1}_{I(0)}[e_i]_{I(0)} \\ & \ge [Z(0)]_{I^c(0)} + |\al| t [R]_{I^c(0)I(0)}[R]^{-1}_{I(0)}[e_i]_I = [Z(t)]_{I^c(0)}.
\end{align*}
So for $j \in I^c(0)$ we get: 
$$
0 \le Z_j(t) \le \ol{Z}_j(t),\ \ t \in [0, \rho_1].
$$
Finally, since $Z_j(t) > 0$ and $\ol{Z}_j(t) > 0$ for $t \in [0, \rho_1)$, we have: $L_j = \ol{L}_j = 0$ on this interval, and trivially
$$
L_j(t) - L_j(s) \ge \ol{L}_j(t) - \ol{L}_j(s),\ 0 \le s \le t \le \rho_1.
$$
The proof is complete. 

\subsection{Proof of Theorem~\ref{main2}.} As in the previous subsection, it suffices to prove the theorem for coupled regular linear driving functions
$$
X(t) = x + \al e_i t,\ \ \ol{X}(t) = \ol{x} + \ol{\al} e_it,
$$
which satisfy the conditions of Theorem~\ref{main2}. This means that $x \le \ol{x}$, and $\al \le \ol{\al}$. Assume the converse: there exist $t \in [0, T]$ and $j = 1, \ldots, N$ such that $Y_j(t) > \ol{Y}_j(t)$. Since $Y_k(0) \le \ol{Y}_k(0)$, $k = 1, \ldots, N$, we can let 
$$
\tau_0 := \inf\{t \ge 0\mid \exists j = 1, \ldots, N: Y_j(t) > \ol{Y}_j(t)\}.
$$
In other words, 
$$
Y_k(\tau_0) \le \ol{Y}_k(\tau_0),\ k = 1, \ldots, N,
$$
but there exists $j = 1, \ldots, N$ such that for every $\eps > 0$ there exists $t \in (\tau_0, \tau_0 + \eps)$ such that $Y_j(t) > \ol{Y}_j(t)$. W.l.o.g. by memoryless property, assume $\tau_0 = 0$. Then $Y_j(0) = \ol{Y}_j(0)$. Recall that $I(t) := \{k = i, \ldots, N\mid Y_k(t) = Y_i(t)\}$, and $\tau_1 := \inf\{t \ge 0\mid I(t) \ne I(0)\}\wedge T$. Define $\ol{I}(t)$ and $\ol{\tau}_1$ similarly for $\ol{Y}$ in place of $Y$. Let $\eps := \tau_1\wedge\ol{\tau}_1$. 

\medskip

{\it Case 1.} $\al \le 0 \le \ol{\al}$. Then $Y_j$ are nonincreasing (follows from Lemma~\ref{explicitCP} and Lemma~\ref{inversionCP}), $\ol{Y}_j$
are nondecreasing, and the statement is trivial.

\medskip

{\it Case 2.} $\al \le \ol{\al} \le 0$. This can be reduced to Case 3 by Lemma~\ref{inversionCP}. 

\medskip

{\it Case 3.} $0 \le \al \le \ol{\al}$. If $j < i$, then particles $Y_j$ and $\ol{Y}_j$ lie below $Y_i(0)$ and $\ol{Y}_i(0)$ respectively. Therefore, they are not hit by the $i$th ranked particles in corresponding systems (which drift upward). In other words, they stay constant: $Y_j(t) = \ol{Y}_j(t) = \const$ on $[0, \eps]$. Now, assume $j \ge i$. Suppose $j \notin I(0)$, that is, $Y_j(0) > Y_i(0)$. Then, again, the particle $Y_j$ is unaffected by $Y_i$ moving upward, at least not until $Y_i$ hits $Y_j$, that is, not until $\tau_1\wedge\ol{\tau}_1$. But the particle $\ol{Y}_j$ is nondecreasing, according to Lemma~\ref{explicitCP}, so $\ol{Y}_j(t) \ge Y_j(t)$ on $[0, \eps)$. Therefore, we are left with the case $j \in I(0)$. Equivalently, $Y_j(0) = Y_i(0)$. And $Y_j(0) = \ol{Y}_j(0) \ge \ol{Y}_i(0)$, so $Y_i(0) \ge \ol{Y}_i(0)$. But by the conditions of the theorem, $Y_i(0) \le \ol{Y}_i(0)$, so  $Y_i(0) = \ol{Y}_i(0)$. Thus, 
$$
Y_i(0) = \ol{Y}_i(0) = Y_j(0) = \ol{Y}_j(0),
$$
and $j \in I(0)\cap \ol{I}(0)$. However, $I(0) \supseteq \ol{I}(0)$, because if $k \in \ol{I}(0)$, then $k \ge i$ and 
$$
Y_k(0) \le \ol{Y}_k(0) = \ol{Y}_i(0) = Y_i(0) \le Y_k(0),
$$
so $Y_k(0) = Y_i(0)$, and $k \in I(0)$. Let $\ol{I}(0) = \{i, \ldots, \ol{k}_0\}$, and $I(0) = \{i, \ldots, k_0\}$. From $\ol{I}(0) \subseteq I(0)$ it follows that $\ol{k}_0 \le k_0$. Therefore, for $t \in [0, \eps]$ we have: 
$$
Y_i(t) \equiv Y_j(t) = Y_i(0) + \al t\left[1 + \frac{q^-_i}{q^+_{i+1}} + \frac{q^-_iq^-_{i+1}}{q^+_{i+1}q^+_{i+2}} + \ldots + \frac{q^-_iq^-_{i+1}\ldots q^-_{k_0 - 1}}{q^+_{i+1}q^+_{i+2}\ldots q^+_{k_0}}\right]^{-1},
$$
$$
\ol{Y}_i(t) \equiv \ol{Y}_j(t) = \ol{Y}_i(0) + \ol{\al} t\left[1 + \frac{\ol{q}^-_i}{\ol{q}^+_{i+1}} + \frac{\ol{q}^-_i\ol{q}^-_{i+1}}{\ol{q}^+_{i+1}\ol{q}^+_{i+2}} + \ldots + \frac{\ol{q}^-_i\ol{q}^-_{i+1}\ldots \ol{q}^-_{\ol{k}_0 - 1}}{\ol{q}^+_{i+1}\ol{q}^+_{i+2}\ldots \ol{q}^+_{k_0}}\right]^{-1}.
$$
But 
$$
\ol{q}^+_k \ge q^+_k,\ \ \ol{q}^-_k \le q^-_k,\ k = 1, \ldots, N;\ \ol{k}_0 \le k_0, 
$$
so 
$$
1 + \frac{q^-_i}{q^+_{i+1}} + \frac{q^-_iq^-_{i+1}}{q^+_{i+1}q^+_{i+2}} + \ldots + \frac{q^-_iq^-_{i+1}\ldots q^-_{k_0 - 1}}{q^+_{i+1}q^+_{i+2}\ldots q^+_{k_0}} \ge 
1 + \frac{\ol{q}^-_i}{\ol{q}^+_{i+1}} + \frac{\ol{q}^-_i\ol{q}^-_{i+1}}{\ol{q}^+_{i+1}\ol{q}^+_{i+2}} + \ldots + \frac{\ol{q}^-_i\ol{q}^-_{i+1}\ldots \ol{q}^-_{\ol{k}_0 - 1}}{\ol{q}^+_{i+1}\ol{q}^+_{i+2}\ldots \ol{q}^+_{k_0}}.
$$
In addition, $\al \le \ol{\al}$, and $Y_i(0) = \ol{Y}_i(0)$. Therefore, we have: 
$Y_i(t) \le \ol{Y}_i(t)$ for $t \in [0, \eps]$, which contradicts our assumption. This completes the proof of Theorem~\ref{main2}. 

\section{Appendix}

\begin{lemma}\label{P1} Take a $d\times d$-reflection nonsingular $\CM$-matrix $R$ and fix a nonempty subset $J \subseteq \{1, \ldots, d\}$. Then 
$$
0 \le [R]_J^{-1} \le [R^{-1}]_J.
$$
\end{lemma}

\begin{proof} Since $R = I_d - Q$, where $Q \ge 0$ is a $d\times d$-matrix with spectral radius strictly less than one, we can apply the Neumann series:
\begin{equation}
\label{611}
R^{-1} = I_d + Q + Q^2 + \ldots
\end{equation}
By Lemma~\ref{P5}, $[R]_J = I_{|J|} - [Q]_J$ is also a reflection nonsingular $\CM$-matrix, so we have:
$$
[R]_J^{-1} = I_{|J|} + [Q]_J + [Q]_J^2 + \ldots
$$
But from~\eqref{611} we get:
$$
[R^{-1}]_J = I_{|J|} + [Q]_J + [Q^2]_J + \ldots
$$
That $[Q^k]_J \ge [Q]_J^k$ for $k = 1, 2, 3, \ldots$ can be proved by induction using Lemma~\ref{P2}. The rest is trivial.
\end{proof}

\begin{lemma}\label{P2} Take nonnegative matrices $A$ ($m\times d$) and $B$ ($d\times n$), and let $I \subseteq \{1, \ldots, m\}$, $J \subseteq \{1, \ldots, d\}$, $K \subseteq \{1, \ldots, n\}$ be nonempty subsets. Then 
$$
[A]_{IJ}[B]_{JK} \le [AB]_{IK}.
$$
\end{lemma}

\begin{proof} Let $A = (a_{ij})$ and $B = (b_{ij})$. Then for $i \in I$ and $k \in K$, 
$$
\left([A]_{IJ}[B]_{JK}\right)_{ik} = \SL_{j \in J}a_{ij}b_{jk} \le \SL_{i=1}^da_{ij}b_{jk} = (AB)_{ik} = \left([AB]_{IK}\right)_{ik}. 
$$
\end{proof}

\begin{lemma}\label{P4} Take a $d\times n$-matrix $A$ and a vector $a \in \BR^n$. Let $I \subseteq \{1, \ldots, d\}$ be a nonempty subset. Then $[Aa]_I = [A]_{I\times\{1, \ldots, n\}}a$. 
\end{lemma}

The proof is trivial.

\begin{lemma}
\label{P7}
Take a $d\times d$-nonnegative matrix $A$ and a nonnegative vector $a \in \BR^d$. Let $J \subseteq \{1, \ldots, d\}$ be a nonempty subset. Then $[Aa]_J \ge [A]_{J}[a]_J$. 
\end{lemma}

The proof is trivial. 

\begin{lemma} \label{P3} Let $R \le \ol{R}$ be two $d\times d$-reflection nonsingular $\CM$-matrices. Then $R^{-1} \ge \ol{R}^{-1} \ge 0$. 
\end{lemma}

\begin{proof} Apply Neumann series again: if 
$$
R = I_d - Q,\ \ \ol{R} = I_d - \ol{Q},
$$
then $Q \ge \ol{Q} \ge 0$, and so $Q^k \ge \ol{Q}^k \ge 0$, $k = 1, 2, \ldots$. Thus,
$$
R^{-1} = I_d + Q + Q^2 + \ldots \ge I_d + \ol{Q} + \ol{Q}^2 + \ldots = \ol{R}^{-1}.
$$
\end{proof}

\begin{lemma}\label{P5} If $R$ is a $d\times d$-reflection nonsingular $\CM$-matrix and $I \subseteq \{1, \ldots, d\}$ is a nonempty subset, then $[R]_I$ is also a reflection nonsingular $\CM$-matrix. 
\end{lemma}

\begin{proof} Use \cite[Lemma 2.1]{MyOwn3}. A $d\times d$-matrix $R = (r_{ij})$ is a reflection nonsingular $\CM$-matrix if and only if 
$$
r_{ii} = 1,\ i = 1, \ldots, d;\ \ r_{ij} \le 0,\ i \ne j,
$$
and, in addition, $R$ is completely-$\CS$, which means that for every principal submatrix $[R]_J$ of $R$ there exists a vector $u > 0$ such that $[R]_Ju > 0$. Now, switch from $R$ to $[R]_I$. The same conditions hold: $$
r_{ii} = 1,\ i \in I;\ \ r_{ij} \le 0,\ i \ne j,\ i, j \in I,
$$
and, in addition, for every principal submatrix $[[R]_I]_J = [R]_J$ of $[R]_I$, where $J \subseteq I$, there exists a vector $u > 0$ such that $[R]_Ju > 0$. This means that $[R]_I$ is also a reflection nonsingular $\CM$-matrix. 
\end{proof}

\begin{lemma} \label{P6} If $A \ge B \ge 0$ and $C \ge D \ge 0$ are matrices such that the matrix products $AC$ and $BD$ are well defined, then $AC \ge BD \ge 0$.
\end{lemma}

The proof is trivial.

\section*{Acknoweldgements}

I would like to thank \textsc{Ioannis Karatzas}, \textsc{Soumik Pal}, \textsc{Xinwei Feng}, \textsc{Amir Dembo} and \textsc{Vladas Sidoravicius} for help and useful discussion. This research was partially supported by  NSF grants DMS 1007563, DMS 1308340, DMS 1405210, and DMS 1409434.

\bibliographystyle{plain}

\bibliography{aggregated}

\begin{thebibliography}{10}

\bibitem{Ichiba11}
Adrian~D. Banner, E.~Robert Fernholz, Tomoyuki Ichiba, Ioannis Karatzas, and
  Vassilios Papathanakos.
\newblock Hybrid atlas models.
\newblock {\em Ann. Appl. Probab.}, 21(2):609--644, 2011.

\bibitem{BFK2005}
Adrian~D. Banner, E.~Robert Fernholz, and Ioannis Karatzas.
\newblock Atlas models of equity markets.
\newblock {\em Ann. Appl. Probab.}, 15(4):2296--2330, 2005.

\bibitem{BG2008}
Adrian~D. Banner and Raouf Ghomrasni.
\newblock Local times of ranked continuous semimartingales.
\newblock {\em Stochastic Process. Appl.}, 118(7):1244--1253, 2008.

\bibitem{CP2010}
Sourav Chatterjee and Soumik Pal.
\newblock A phase transition behavior for {B}rownian motions interacting
  through their ranks.
\newblock {\em Probab. Theory Related Fields}, 147(1-2):123--159, 2010.

\bibitem{4people}
Amir Dembo, Mykhaylo Shkolnikov, S.R.~Srinivasa Varadhan, and Ofer Zeitouni.
\newblock Large deviations for diffusions interacting through their ranks.
\newblock 2013.
\newblock Preprint. Available at arXiv:1211.5223.

\bibitem{DemboTsai}
Amir Dembo and Li-Cheng Tsai.
\newblock Equilibrium fluctuation of the atlas model.
\newblock 2015.
\newblock Preprint. Available at arXiv:1503.03581.

\bibitem{FIK2013}
E.~Robert Fernholz, Tomoyuki Ichiba, and Ioannis Karatzas.
\newblock Two {B}rownian particles with rank-based characteristics and
  skew-elastic collisions.
\newblock {\em Stochastic Process. Appl.}, 123(8):2999--3026, 2013.

\bibitem{Haddad2010}
Jean-Paul Haddad, Ravi~R. Mazumdar, and Francisco~J. Piera.
\newblock Pathwise comparison results for stochastic fluid networks.
\newblock {\em Queueing Syst.}, 66(2):155--168, 2010.

\bibitem{Har1973}
J.~Michael Harrison.
\newblock The heavy traffic approximation for single server queues in series.
\newblock {\em J. Appl. Probability}, 10:613--629, 1973.

\bibitem{Har1978}
J.~Michael Harrison.
\newblock The diffusion approximation for tandem queues in heavy traffic.
\newblock {\em Adv. in Appl. Probab.}, 10(4):886--905, 1978.

\bibitem{HR1981a}
J.~Michael Harrison and I.~Martin Reiman.
\newblock Reflected {B}rownian motion on an orthant.
\newblock {\em Ann. Probab.}, 9(2):302--308, 1981.

\bibitem{HW1987b}
J.~Michael Harrison and R.~J. Williams.
\newblock Brownian models of open queueing networks with homogeneous customer
  populations.
\newblock {\em Stochastics}, 22(2):77--115, 1987.

\bibitem{IchibaThesis}
Tomoyuki Ichiba.
\newblock {\em Topics in multi-dimensional diffusion theory: {A}ttainability,
  reflection, ergodicity and rankings}.
\newblock ProQuest LLC, Ann Arbor, MI, 2009.
\newblock Thesis (Ph.D.)--Columbia University.

\bibitem{IK2010}
Tomoyuki Ichiba and Ioannis Karatzas.
\newblock On collisions of {B}rownian particles.
\newblock {\em Ann. Appl. Probab.}, 20(3):951--977, 2010.

\bibitem{IKS2013}
Tomoyuki Ichiba, Ioannis Karatzas, and Mykhaylo Shkolnikov.
\newblock Strong solutions of stochastic equations with rank-based
  coefficients.
\newblock {\em Probab. Theory Related Fields}, 156(1-2):229--248, 2013.

\bibitem{IPS2012}
Tomoyuki Ichiba, Soumik Pal, and Mykhaylo Shkolnikov.
\newblock Convergence rates for rank-based models with applications to
  portfolio theory.
\newblock {\em Probability Theory and Related Fields}, pages 1--34, 2012.

\bibitem{Jourdain2000}
B.~Jourdain.
\newblock Diffusion processes associated with nonlinear evolution equations for
  signed measures.
\newblock {\em Methodol. Comput. Appl. Probab.}, 2(1):69--91, 2000.

\bibitem{JM2008}
Benjamin Jourdain and Florent Malrieu.
\newblock Propagation of chaos and poincaré inequalities for a system of
  particles interacting through their cdf.
\newblock {\em Ann. Appl. Probab.}, 18(5):1706--1736, 10 2008.

\bibitem{JR2013b}
Benjamin Jourdain and Julien Reygner.
\newblock Capital distribution and portfolio performance for in the mean-field
  atlas model.
\newblock 2013.
\newblock Preprint. Available at arXiv:1312.5660.

\bibitem{JR2013a}
Benjamin Jourdain and Julien Reygner.
\newblock Propagation of chaos for rank-based interacting diffusions and long
  time behaviour of a scalar quasilinear parabolic equation.
\newblock {\em Stochastic Partial Differential Equations: Analysis and
  Computations}, 1(3):455--506, 2013.

\bibitem{JR2014}
Benjamin Jourdain and Julien Reygner.
\newblock The small noise limit of order-based diffusion processes.
\newblock {\em Electron. J. Probab.}, 19, 2014.

\bibitem{KPS2012}
Ioannis Karatzas, Soumik Pal, and Mykhaylo Shkolnikov.
\newblock Systems of brownian particles with asymmetric collisions.
\newblock 2012.
\newblock Preprint. Available at arXiv:1210.0259v1.

\bibitem{MyOwn4}
Ioannis Karatzas and Andrey Sarantsev.
\newblock Diverse market models of competing brownian particles with splits and
  mergers.
\newblock 2014.
\newblock Preprint. Available at arXiv:1404.0748.

\bibitem{KR2012b}
Offer Kella and S.~Ramasubramanian.
\newblock Asymptotic irrelevance of initial conditions for {S}korohod
  reflection mapping on the nonnegative orthant.
\newblock {\em Math. Oper. Res.}, 37(2):301--312, 2012.

\bibitem{KW1996}
Offer Kella and Ward Whitt.
\newblock Stability and structural properties of stochastic storage networks.
\newblock {\em J. Appl. Probab.}, 33(4):1169--1180, 1996.

\bibitem{PP2008}
Soumik Pal and Jim Pitman.
\newblock One-dimensional {B}rownian particle systems with rank-dependent
  drifts.
\newblock {\em Ann. Appl. Probab.}, 18(6):2179--2207, 2008.

\bibitem{PS2010}
Soumik Pal and Mykhaylo Shkolnikov.
\newblock Concentration of measure for brownian particle systems interacting
  through their ranks.
\newblock {\em Ann. Appl. Probab.}, 24(4):1482--1508, 08 2014.

\bibitem{R2000}
S.~Ramasubramanian.
\newblock A subsidy-surplus model and the {S}korokhod problem in an orthant.
\newblock {\em Math. Oper. Res.}, 25(3):509--538, 2000.

\bibitem{Reygner2014}
Julien Reygner.
\newblock Chaoticity of the stationary distribution of rank-based interacting
  diffusions.
\newblock 2014.
\newblock Preprint. Available at arXiv:1408.4103.

\bibitem{MyOwn6}
Andrey Sarantsev.
\newblock Infinite systems of competing brownian particles.
\newblock 2015.
\newblock Preprint. Available at arXiv:1403.4229.

\bibitem{MyOwn5}
Andrey Sarantsev.
\newblock Multiple collisions in systems of competing brownian particles.
\newblock 2015.
\newblock Preprint. Available at arXiv:1309.2621.

\bibitem{MyOwn3}
Andrey Sarantsev.
\newblock Triple and simultaneous collisions of competing brownian particles.
\newblock {\em Electron. J. Probab.}, 20:no. 29, 1--28, 2015.

\bibitem{MyOwn11}
Andrey Sarantsev.
\newblock Two-sided infinite systems of competing brownian particles.
\newblock 2015.
\newblock Preprint. Available at arXiv:1509.01859.

\bibitem{S2011}
Mykhaylo Shkolnikov.
\newblock Competing particle systems evolving by interacting {L}\'evy
  processes.
\newblock {\em Ann. Appl. Probab.}, 21(5):1911--1932, 2011.

\bibitem{S2012}
Mykhaylo Shkolnikov.
\newblock Large systems of diffusions interacting through their ranks.
\newblock {\em Stochastic Process. Appl.}, 122(4):1730--1747, 2012.

\bibitem{Wil1995}
R.~J. Williams.
\newblock Semimartingale reflecting {B}rownian motions in the orthant.
\newblock In {\em Stochastic networks}, volume~71 of {\em IMA Vol. Math.
  Appl.}, pages 125--137. Springer, New York, 1995.

\end{thebibliography}

\end{document}